\documentclass[preprint,10pt]{elsarticle}

\usepackage{amsmath}
\usepackage{amssymb}
\usepackage{amsthm}

\newtheorem{thm}{Theorem}[section]
\newtheorem{cor}[thm]{Corollary}
\newtheorem{lem}[thm]{Lemma}
\newtheorem{prop}[thm]{Proposition}
\newtheorem{defin}[thm]{Definition}
\newtheorem{rem}[thm]{Remark}
\numberwithin{equation}{section}

\journal{}

\begin{document}

\begin{frontmatter}



\title{Homogeneous fractional integral operators on Lebesgue and Morrey spaces, Hardy--Littlewood--Sobolev and Olsen-type inequalities}


\author{Kaikai Yang and Hua Wang \footnote{In memory of Li Xue. E-mail address: 1362447069@qq.com, wanghua@pku.edu.cn.}}
\address{School of Mathematics and Systems Science, Xinjiang University,\\
Urumqi 830046, P. R. China}

\begin{abstract}
Let $T_{\Omega,\alpha}$ be the homogeneous fractional integral operator defined as
\begin{equation*}
T_{\Omega,\alpha}f(x):=\int_{\mathbb R^n}\frac{\Omega(x-y)}{|x-y|^{n-\alpha}}f(y)\,dy,
\end{equation*}
and the related fractional maximal operator $M_{\Omega,\alpha}$ is given by
\begin{equation*}
M_{\Omega,\alpha}f(x):=\sup_{r>0}\frac{1}{|B(x,r)|^{1-\alpha/n}}\int_{|x-y|<r}|\Omega(x-y)f(y)|\,dy.
\end{equation*}
In this article, we will use the idea of Hedberg to reprove that the operators $T_{\Omega,\alpha}$ and $M_{\Omega,\alpha}$ are bounded from $L^p(\mathbb R^n)$ to $L^q(\mathbb R^n)$ provided that $\Omega\in L^s(\mathbf{S}^{n-1})$, $s'<p<n/{\alpha}$ and $1/q=1/p-{\alpha}/n$, which was obtained by Muckenhoupt and Wheeden. We also reprove that under the assumptions that $\Omega\in L^s(\mathbf{S}^{n-1})$, $s'\leq p<n/{\alpha}$ and $1/q=1/p-{\alpha}/n$, the operators $T_{\Omega,\alpha}$ and $M_{\Omega,\alpha}$ are bounded from $L^p(\mathbb R^n)$ to $L^{q,\infty}(\mathbb R^n)$, which was obtained by Chanillo, Watson and Wheeden. We will use the idea of Adams to show that $T_{\Omega,\alpha}$ and $M_{\Omega,\alpha}$ are bounded from $L^{p,\kappa}(\mathbb R^n)$ to $L^{q,\kappa}(\mathbb R^n)$ whenever $s'<p<n/{\alpha}$ and $1/q=1/p-\alpha/{n(1-\kappa)}$, and bounded from $L^{p,\kappa}(\mathbb R^n)$ to $WL^{q,\kappa}(\mathbb R^n)$ whenever $s'\leq p<n/{\alpha}$ and $1/q=1/p-\alpha/{n(1-\kappa)}$. Some new estimates in the limiting cases are also established. The results obtained are substantial improvements and extensions of some known results. Moreover, we will apply these results to several well-known inequalities such as Hardy--Littlewood--Sobolev and Olsen-type inequalities.
\end{abstract}

\begin{keyword}
Fractional integral operators, fractional maximal operators, homogeneous kernel, Morrey spaces, Hardy--Littlewood--Sobolev inequalities, Olsen-type inequalities
\MSC[2020] Primary 42B20, 42B25; Secondary 42B35
\end{keyword}

\end{frontmatter}

\section{Introduction}
\label{sec1}
Let $\mathbb R^n$ be the $n$-dimensional Euclidean space endowed with the Lebesgue measure $dx$ and the Euclidean norm $|\cdot|$. Given $\alpha\in(0,n)$, the Riesz potential operator $I_{\alpha}$ of order $\alpha$ (also referred to as the fractional integral operator) is defined by
\begin{equation*}
I_{\alpha}f(x):=\frac{1}{\gamma(\alpha)}\int_{\mathbb R^n}\frac{f(y)}{|x-y|^{n-\alpha}}dy,\quad x\in\mathbb R^n,
\end{equation*}
where $\gamma(\alpha):=\frac{2^{\alpha}\pi^{n/2}\Gamma(\alpha/2)}{\Gamma({(n-\alpha)}/2)}$ and $\Gamma(\cdot)$ being the usual gamma function. Let $(-\Delta)^{\alpha/2}$ (under $0<\alpha<n$) denote the $\alpha/2$-th order Laplacian. Then $u=I_{\alpha}f$ is viewed as a solution of the $\alpha/2$-th order Laplace equation
\begin{equation*}
(-\Delta)^{\alpha/2}u=f
\end{equation*}
in the sense of the Fourier transform; i.e., $(-\Delta)^{\alpha/2}$ exists as the inverse of $I_{\alpha}$. It is well known that the Riesz potential operator $I_{\alpha}$ plays an important role in harmonic analysis, PDE and potential theory, particularly in the study of smoothness properties of functions. Let $1<p<q<\infty$. The classical Hardy--Littlewood--Sobolev theorem states that $I_{\alpha}$ is bounded from $L^p(\mathbb R^n)$ to $L^q(\mathbb R^n)$ if and only if $1/q=1/p-\alpha/n$. We also know that for $1=p<q<\infty$, $I_{\alpha}$ is bounded from $L^1(\mathbb R^n)$ to $L^{q,\infty}(\mathbb R^n)$ if and only if $q=n/{(n-\alpha)}$. However, for the critical index $p=n/{\alpha}$ and $0<\alpha<n$, the strong type $(p,\infty)$ estimate of the operator $I_{\alpha}$ does not hold in general (see \cite[p.119]{stein}). Instead, in this case, we have that as a substitute the Riesz potential operator $I_{\alpha}$ is bounded from $L^{n/{\alpha}}(\mathbb R^n)$ to $\mathrm{BMO}(\mathbb R^n)$ (see \cite[p.130]{duoand}). For each locally integrable function $f$, closely related to the Riesz potential operator is the fractional maximal operator $M_{\alpha}$ on $\mathbb R^n$, which is defined by
\begin{equation*}
M_{\alpha}f(x):=\sup_{r>0}\frac{1}{m(B(x,r))^{1-\alpha/n}}\int_{|x-y|<r}|f(y)|\,dy,\quad x\in\mathbb R^n,
\end{equation*}
where the supremum is taken over all $r>0$. For $x_0\in\mathbb R^n$ and $r>0$, let $B(x_0,r)=\{x\in\mathbb R^n:|x-x_0|<r\}$ denote the open ball centered at $x_0$ of radius $r$, $B(x_0,r)^{\complement}$ denote its complement and $m(B(x_0,r))$ be the Lebesgue measure of the ball $B(x_0,r)$. It can be shown that $M_{\alpha}$ satisfies the same norm inequalities as $I_{\alpha}$ for $0<\alpha<n$, $1<p<n/{\alpha}$ and $1/q=1/p-\alpha/n$. The weak type $(1,q)$ inequality can be proved by using covering lemma arguments when $q=n/{(n-\alpha)}$ (see, for example, \cite[Theorem 2]{muckenhoupt2}), and $M_{\alpha}$ is clearly of strong type $(p,\infty)$ when $p=n/{\alpha}$. Actually, by H\"{o}lder's inequality
\begin{equation*}
\frac{1}{m(B(x,r))^{1-\alpha/n}}\int_{|x-y|<r}|f(y)|\,dy\leq\bigg(\int_{|x-y|<r}|f(y)|^p\,dy\bigg)^{1/p}\leq\|f\|_{L^p}.
\end{equation*}
The strong type $(p,q)$ inequalities are established via the Marcinkiewicz interpolation theorem. Recall that, for any given $p\in[1,\infty)$, the space $L^p(\mathbb R^n)$ is defined as the set of all measurable functions $f$ on $\mathbb R^n$ such that
\begin{equation*}
\|f\|_{L^p}:=\bigg(\int_{\mathbb R^n}|f(x)|^p\,dx\bigg)^{1/p}<\infty,
\end{equation*}
and the space $L^{p,\infty}(\mathbb R^n)$ is defined as the set of all measurable functions $f$ on $\mathbb R^n$ such that
\begin{equation*}
\|f\|_{L^{p,\infty}}:=\sup_{\lambda>0}\lambda\cdot m\big(\big\{x\in\mathbb R^n:|f(x)|>\lambda\big\}\big)^{1/p}<\infty.
\end{equation*}
Let $L^{\infty}(\mathbb R^n)$ denote the Banach space of all essentially bounded measurable functions $f$ on $\mathbb R^n$. A locally integrable function $f$ is said to be in the bounded mean oscillation space $\mathrm{BMO}(\mathbb R^n)$ (see \cite{john}), if
\begin{equation*}
\|f\|_{\mathrm{BMO}}:=\sup_{B}\frac{1}{m(B)}\int_B|f(x)-f_B|\,dx<\infty,
\end{equation*}
where $f_B$ denotes the average of $f$ on $B$; i.e., $f_B:=\frac{1}{m(B)}\int_B f(y)\,dy$ and the supremum is taken over all balls $B$ in $\mathbb R^n$. Modulo constants, the space $\mathrm{BMO}(\mathbb R^n)$ is a Banach space with respect to the norm $\|\cdot\|_{\mathrm{BMO}}$.

Let $\mathbf{S}^{n-1}=\{x\in\mathbb R^n:|x|=1\}$ denote the unit sphere in $\mathbb R^n$($n\ge2$) equipped with the normalized Lebesgue measure $d\sigma(x')$. Let $\Omega$ be a homogeneous function of degree zero on $\mathbb R^n$ and $\Omega\in L^s(\mathbf{S}^{n-1})$ with $s\geq1$. Then the homogeneous fractional integral operator $T_{\Omega,\alpha}$ is defined by
\begin{equation*}
T_{\Omega,\alpha}f(x):=\int_{\mathbb R^n}\frac{\Omega(x-y)}{|x-y|^{n-\alpha}}f(y)\,dy,\quad x\in\mathbb R^n.
\end{equation*}
Given any $\alpha\in(0,n)$ and $f\in L^1_{\mathrm{loc}}(\mathbb R^n)$, the fractional maximal function of $f$ is defined by
\begin{equation*}
M_{\Omega,\alpha}f(x):=\sup_{r>0}\frac{1}{m(B(x,r))^{1-\alpha/n}}\int_{|x-y|<r}|\Omega(x-y)f(y)|\,dy,\quad x\in\mathbb R^n,
\end{equation*}
where the supremum on the right-hand side is taken over all $r>0$.
\begin{description}
  \item[1] When $\alpha=0$ and $\Omega\equiv1$, $M_{\Omega,\alpha}$ is the standard Hardy--Littlewood maximal operator on $\mathbb R^n$, which is denoted by $M$;
  \item[2] When $\Omega\equiv1$, $M_{\Omega,\alpha}$ is just the fractional maximal operator $M_{\alpha}$ on $\mathbb R^n$, and $T_{\Omega,\alpha}$ is just the Riesz potential operator $I_{\alpha}$.
\end{description}
In the present work, we study the homogeneous fractional integral operators and fractional maximal operators acting on Lebesgue and Morrey spaces.
In \cite{muckenhoupt1}, Muckenhoupt and Wheeden proved that $T_{\Omega,\alpha}$ is bounded from $L^p(\mathbb R^n)$ to $L^q(\mathbb R^n)$, when $\Omega\in L^s(\mathbf{S}^{n-1})$ with $s>p'$, $1<p<n/{\alpha}$ and $1/q=1/p-{\alpha}/n$ (see also \cite{ding1,lu} for the weighted version). We will give an alternative proof of this result, following the idea due to Hedberg in \cite{hed}. In \cite{cha},  Chanillo, Watson and Wheeden showed that $T_{\Omega,\alpha}$ is of weak type $(1,n/{(n-\alpha)})$ when $\Omega\in L^s(\mathbf{S}^{n-1})$ with $s\geq n/{(n-\alpha)}$. We will show that $T_{\Omega,\alpha}$ is of weak type $(p,q)$ when $\Omega\in L^s(\mathbf{S}^{n-1})$ with $s'\leq p<n/{\alpha}$ and $1/q=1/p-{\alpha}/n$. For Riesz potential operators $I_{\alpha}$($0<\alpha<n$), Hedberg \cite{hed} established the following pointwise inequality:
\begin{equation*}
|I_{\alpha}f(x)|\leq C\big[\|f\|_{L^p}\big]^{1-p/q}\big[Mf(x)\big]^{p/q},\quad x\in\mathbb R^n,
\end{equation*}
for any $f\in L^p(\mathbb R^n)$ with $1\leq p<n/{\alpha}$ and $1/q=1/p-{\alpha}/n$. For homogeneous fractional integral operators $T_{\Omega,\alpha}$, we will establish the corresponding estimate
\begin{equation*}
\begin{split}
|T_{\Omega,\alpha}f(x)|\leq C\big[\|f\|_{L^p}\big]^{1-p/q}\big[M_{s'}f(x)\big]^{p/q},\quad x\in\mathbb R^n,
\end{split}
\end{equation*}
where the operator $M_{s'}$ is given by $M_{s'}(f)=[M(|f|^{s'})]^{1/{s'}}$. Based on this pointwise domination and the boundedness of $M$, we can prove norm inequalities involving fractional maximal functions and fractional integral operators. Moreover, in \cite{ding2}, Ding and Lu considered the critical case $p=n/{\alpha}$, and proved that if $\Omega$ satisfies certain smoothness condition $\mathfrak{D}_{s}$ with $s>n/{(n-\alpha)}$(see Definition \ref{defd} below), then $T_{\Omega,\alpha}$ is bounded from $L^{p}(\mathbb R^n)$ to $\mathrm{BMO}(\mathbb R^n)$. For the critical index $p=n/{\alpha}$, we will show that this result also holds for $s=n/{(n-\alpha)}$.

On the other hand, the classical Morrey spaces $L^{p,\kappa}(\mathbb R^n)$ were originally introduced by Morrey in \cite{morrey} to study the local regularity of solutions to second order elliptic partial differential equations. Nowadays these spaces have been studied intensively in the literature, and found a wide range of applications in harmonic analysis, potential theory and nonlinear dispersive equations. Let us first recall the definition of the Morrey space. Let $1\leq p<\infty$ and $0\leq\kappa\leq1$. We denote by $L^{p,\kappa}(\mathbb R^n)$ the Morrey space of all $p$-locally integrable functions $f$ on $\mathbb R^n$ such that
\begin{equation*}
\begin{split}
\|f\|_{L^{p,\kappa}}:=&\sup_{B}\bigg(\frac{1}{m(B)^{\kappa}}\int_B|f(x)|^p\,dx\bigg)^{1/p}\\
=&\sup_{B}\frac{1}{m(B)^{\kappa/p}}\big\|f\cdot\chi_{B}\big\|_{L^p}<\infty.
\end{split}
\end{equation*}
Note that $L^{p,0}(\mathbb R^n)=L^p(\mathbb R^n)$ and $L^{p,1}(\mathbb R^n)=L^{\infty}(\mathbb R^n)$ by the Lebesgue
differentiation theorem. If $\kappa<0$ or $\kappa>1$, then $L^{p,\kappa}=\Theta$, where $\Theta$ is the set of all functions equivalent to $0$ on $\mathbb R^n$. We also denote by $WL^{p,\kappa}(\mathbb R^n)$ the weak Morrey space of all measurable functions $f$ on $\mathbb R^n$ such that
\begin{equation*}
\begin{split}
\|f\|_{WL^{p,\kappa}}:=&\sup_{B}\frac{1}{m(B)^{\kappa/p}}\big\|f\cdot\chi_{B}\big\|_{L^{p,\infty}}\\
=&\sup_{B}\frac{1}{m(B)^{\kappa/p}}\sup_{\lambda>0}\lambda\cdot m\big(\big\{x\in B:|f(x)|>\lambda\big\}\big)^{1/p}<\infty.
\end{split}
\end{equation*}
In \cite{peetre}, Peetre established the boundedness properties of the Riesz potential operator $I_{\alpha}$ and fractional maximal operator $M_{\alpha}$ on Morrey spaces. His results can be summarized as follows.

\newtheorem*{thma}{Theorem A}
\begin{thma}[\cite{peetre}]
Let $0<\alpha<n$, $1<p<n/{\alpha}$ and $1/q=1/p-\alpha/n$. If $0<\kappa<p/q$, then the fractional maximal operator $M_\alpha$ and the Riesz potential operator $I_{\alpha}$ are bounded from $L^{p,\kappa}(\mathbb R^n)$ to $L^{q,\kappa^{*}}(\mathbb R^n)$ with $\kappa^{*}:={(\kappa q)}/p$. Moreover, if $p=1$ and $0<\kappa<1/q$, then the operators $M_{\alpha}$ and $I_{\alpha}$ are bounded from $L^{1,\kappa}(\mathbb R^n)$ to $WL^{q,\kappa^{*}}(\mathbb R^n)$ with $q:=n/{(n-\alpha)}$ and $\kappa^{*}:=\kappa q$.
\end{thma}
In \cite{adams}, Adams proved that for any $x\in\mathbb R^n$ and $0<\alpha<n$, the following estimate holds.
\begin{equation*}
\begin{split}
|I_{\alpha}f(x)|&\leq C\big[Mf(x)\big]^{p/q}\cdot\big[\|f\|_{L^{p,\kappa}}\big]^{1-p/q},
\end{split}
\end{equation*}
for all $f\in L^{p,\kappa}(\mathbb R^n)$ with $1\leq p<n/{\alpha}$ and $1/q=1/p-\alpha/{[n(1-\kappa)]}$. A simple calculation shows that $M_{\alpha}(f)$ can be controlled pointwise by $I_{\alpha}(|f|)$ for any function $f$. Indeed, for any $0<\alpha<n$, $x\in\mathbb R^n$ and $r>0$, we have
\begin{equation*}
\begin{split}
I_{\alpha}(|f|)(x)&\geq\frac{1}{\gamma(\alpha)}\int_{|y-x|\leq r}\frac{|f(y)|}{|x-y|^{n-\alpha}}\,dy\\
&\geq \frac{1}{\gamma(\alpha)}\cdot\frac{1}{r^{n-\alpha}}\int_{|y-x|\leq r}|f(y)|\,dy.
\end{split}
\end{equation*}
Taking the supremum for all $r>0$ on both sides of the above inequality, we get
\begin{equation*}
M_{\alpha}(f)(x)\leq \frac{\gamma(\alpha)}{(v_n)^{1-\alpha/n}}\cdot I_{\alpha}(|f|)(x),\quad\mbox{for all}\; x\in\mathbb R^n,
\end{equation*}
where $v_n$ is the volume of the unit ball in $\mathbb R^n$ (However, the converse inequality does not hold in general). In view of the above estimates, Adams obtained the following results.
\newtheorem*{thmb}{Theorem B}
\begin{thmb}[\cite{adams}]
Let $0<\alpha<n$ and $1<p<n/{\alpha}$. If $0<\kappa<1-{(\alpha p)}/n$ and $1/q=1/p-\alpha/{[n(1-\kappa)]}$, then the fractional maximal operator $M_\alpha$ and the Riesz potential operator $I_{\alpha}$ are bounded from $L^{p,\kappa}(\mathbb R^n)$ to $L^{q,\kappa}(\mathbb R^n)$. Moreover, if $p=1$ and $0<\kappa<1-\alpha/n$, then the operators $M_{\alpha}$ and $I_{\alpha}$ are bounded from $L^{1,\kappa}(\mathbb R^n)$ to $WL^{q,\kappa}(\mathbb R^n)$ with $1/q:=1-\alpha/{[n(1-\kappa)]}$.
\end{thmb}
\begin{rem}
It is worth pointing out that Theorem B is an improvement of Theorem A. It can be shown that for $(q,\kappa)\in(1,\infty)\times(0,1)$, the inclusion relation
\begin{equation}\label{inclusion1}
L^{q,\kappa}(\mathbb R^n)\hookrightarrow L^{q^{*},\kappa^{*}}(\mathbb R^n)
\end{equation}
holds for $q^{*}<q$ and $1-\kappa^{*}=(1-\kappa)(q^{*}/q)$, and $\|f\|_{L^{q^{*},\kappa^{*}}}\leq\|f\|_{L^{q,\kappa}}$. Indeed, by H\"{o}lder's inequality, we know that for any ball $B$ in $\mathbb R^n$ and $1\leq q^{*}<q<\infty$,
\begin{equation*}
\begin{split}
&\bigg(\frac{1}{m(B)^{\kappa^{*}}}\int_B|f(x)|^{q^*}\,dx\bigg)^{1/{q^*}}\\
&=m(B)^{{(1-\kappa^*)}/{q^*}}\bigg(\frac{1}{m(B)}\int_B|f(x)|^{q^*}\,dx\bigg)^{1/{q^*}}\\
&\leq m(B)^{{(1-\kappa^*)}/{q^*}}\bigg(\frac{1}{m(B)}\int_B|f(x)|^{q}\,dx\bigg)^{1/{q}}\\
&=\bigg(\frac{1}{m(B)^{\kappa}}\int_B|f(x)|^{q}\,dx\bigg)^{1/{q}},
\end{split}
\end{equation*}
where in the last equality we have used the fact that $\kappa/q=1/q-{(1-\kappa^*)}/{q^*}$. This gives \eqref{inclusion1}. Moreover, using H\"{o}lder's inequality for weak $L^q$ spaces and repeating the argument above, we can also obtain
\begin{equation}\label{inclusion2}
WL^{q,\kappa}(\mathbb R^n)\hookrightarrow WL^{q^{*},\kappa^{*}}(\mathbb R^n)
\end{equation}
whenever $1\leq q^{*}<q<\infty$ and $1-\kappa^{*}=(1-\kappa)(q^{*}/q)$. Thus, by the conditions of Theorem B, we can see that for $f\in L^{p,\kappa}(\mathbb R^n)$,
\begin{equation*}
I_{\alpha}(f)(\mathrm{or}~ M_{\alpha}(f))\in L^{q,\kappa}(\mathbb R^n)\Longrightarrow
I_{\alpha}(f)(\mathrm{or}~ M_{\alpha}(f))\in L^{q^{*},\kappa^{*}}(\mathbb R^n)
\end{equation*}
whenever $1/{q^*}=1/p-\alpha/n$ and $\kappa^{*}={(\kappa q^{*})}/p$, and for $f\in L^{1,\kappa}(\mathbb R^n)$,
\begin{equation*}
I_{\alpha}(f)(\mathrm{or}~ M_{\alpha}(f))\in WL^{q,\kappa}(\mathbb R^n)\Longrightarrow
I_{\alpha}(f)(\mathrm{or}~ M_{\alpha}(f))\in WL^{q^{*},\kappa^{*}}(\mathbb R^n)
\end{equation*}
whenever $q^{*}=n/{(n-\alpha)}$ and $\kappa^{*}=\kappa {q^*}$.
\end{rem}
We consider fractional integral operators with homogeneous kernels on Morrey spaces and give the following result, which has been proved by Lu, Yang and Zhou in \cite{lu2}(see also \cite{mizuhara} and \cite{wang} for the weighted case).
\newtheorem*{thmc}{Theorem C}
\begin{thmc}[\cite{lu2}]
Suppose that $\Omega\in L^s(\mathbf{S}^{n-1})$ with $1<s\leq\infty$. Let $0<\alpha<n$, $s'<p<n/{\alpha}$ and $1/q=1/p-{\alpha}/n$. If $0<\kappa<p/q$, then the fractional integral operator $T_{\Omega,\alpha}$ with homogeneous kernel is bounded from $L^{p,\kappa}(\mathbb R^n)$ to $L^{q,\kappa^{*}}(\mathbb R^n)$ with $\kappa^{*}={(\kappa q)}/p$.
\end{thmc}
Using the same argument as above, we can also show that $M_{\Omega,\alpha}(f)$ can be controlled pointwise by $T_{\Omega,\alpha}(|f|)$ for an arbitrary function $f$ and $0<\alpha<n$. Hence, the same conclusion of Theorem C also holds for the fractional maximal operator $M_{\Omega,\alpha}$.
Theorem C extends the result of Peetre for the Riesz potential operators.

Motivated by Theorems B and C, it is natural to ask whether the result of Adams holds for the operators $M_{\Omega,\alpha}$ and $T_{\Omega,\alpha}$, provided that $\Omega\in L^s(\mathbf{S}^{n-1})$ with $1<s\leq\infty$. Our aim in this paper is to prove the homogeneous fractional integral operators related to the Adams inequality on Morrey spaces. Following the same arguments as in \cite{adams,adams1}, we will prove
\begin{equation*}
\begin{split}
\big|T_{\Omega,\alpha}f(x)\big|&\leq C\big[M_{s'}f(x)\big]^{p/q}\cdot\big[\|f\|_{L^{p,\kappa}}\big]^{1-p/q},\quad x\in\mathbb R^n,
\end{split}
\end{equation*}
for any $f\in L^{p,\kappa}(\mathbb R^n)$ with $s'\leq p<n/{\alpha}$ and $1/q=1/p-\alpha/{n(1-\kappa)}$. From this and the boundedness of $M$, the desired estimate follows. This improves previous results of Theorem C. Moreover, some new estimates in the limiting cases are also discussed. Finally, we will apply our main results to several well-known inequalities on Euclidean space $\mathbb R^n$ such as Hardy--Littlewood--Sobolev and Olsen-type inequalities.

In this article, $C>0$ denotes a universal constant which is independent of the main parameters and may change from line to line. The symbol $\mathbf{X}\lesssim \mathbf{Y}$ means that $\mathbf{X}\leq C\mathbf{Y}$.
\section{Main results}
\label{sec2}
In this section, we will establish boundedness properties of homogeneous fractional integral operators $T_{\Omega,\alpha}$ and fractional maximal operators $M_{\Omega,\alpha}$ on Lebesgue and Morrey spaces.

\subsection{Boundedness of $T_{\Omega,\alpha}$ and $M_{\Omega,\alpha}$ on Lebesgue spaces}
\label{sec21}
\begin{thm}\label{thm1}
Let $\Omega\in L^s(\mathbf{S}^{n-1})$ with $1<s\leq\infty$. Let $0<\alpha<n$, $1\leq p<q<\infty$ and $1/q=1/p-\alpha/n$.

$(i)$ If $s'=p$, then for every $p<n/{\alpha}$, the inequality
\begin{equation*}
\|T_{\Omega,\alpha}(f)\|_{WL^q}\lesssim\|f\|_{L^p}
\end{equation*}
holds.

$(ii)$ If $s'<p$, then for every $p\in(s',n/{\alpha})$, the inequality
\begin{equation*}
\|T_{\Omega,\alpha}(f)\|_{L^q}\lesssim\|f\|_{L^p}
\end{equation*}
holds.
\end{thm}
\begin{proof}[Proof of Theorem \ref{thm1}]
The idea of the proof is due to Hedberg \cite{hed}. For given $L^p(\mathbb R^n)$ with $1\leq p<n/{\alpha}$, he decomposes the integral defining $I_{\alpha}$ into two parts, one over $\{x:|x-y|<\sigma\}$ and the other over $\{x:|x-y|\geq \sigma\}$. Using the definition of the maximal operator $M$ to the first part and applying the H\"{o}lder inequality to the second. Now choose $\sigma>0$ so that both terms are the same size. Then the following inequality is obtained.
\begin{equation*}
|I_{\alpha}f(x)|\leq C\big[\|f\|_{L^p}\big]^{1-p/q}\cdot\big[Mf(x)\big]^{p/q},\quad x\in\mathbb R^n.
\end{equation*}
Arguing as in \cite{hed}, we will prove the following result.
\begin{equation}\label{wanghua1}
\begin{split}
\big|T_{\Omega,\alpha}f(x)\big|\lesssim\big[\|f\|_{L^p}\big]^{1-p/q}\cdot\big[M_{s'}f(x)\big]^{p/q},\quad x\in\mathbb R^n.
\end{split}
\end{equation}
Let $f\in L^{p}(\mathbb R^n)$ with $s'\leq p<n/{\alpha}$. In order to prove \eqref{wanghua1}, one can write
\begin{equation*}
\begin{split}
\big|T_{\Omega,\alpha}f(x)\big|&\leq \int_{\mathbb R^n}\frac{|\Omega(x-y)|}{|x-y|^{n-\alpha}}|f(y)|\,dy\\
&=\int_{|x-y|<\sigma}\frac{|\Omega(x-y)|}{|x-y|^{n-\alpha}}|f(y)|\,dy
+\int_{|x-y|\geq\sigma}\frac{|\Omega(x-y)|}{|x-y|^{n-\alpha}}|f(y)|\,dy\\
&:=\mathrm{I+II},
\end{split}
\end{equation*}
where $\sigma$ is a positive constant to be fixed below. Let us analyze the first term $\mathrm{I}$. For any $x\in\mathbb{R}^n$, we observe that
\begin{equation*}
\begin{split}
\mathrm{I}&=\sum_{j=1}^\infty\int_{2^{-j}\sigma\leq|x-y|<2^{-j+1}\sigma}\frac{|\Omega(x-y)|}{|x-y|^{n-\alpha}}|f(y)|\,dy\\
&\leq\sum_{j=1}^\infty\frac{1}{(2^{-j}\sigma)^{n-\alpha}}\int_{2^{-j}\sigma\leq|x-y|<2^{-j+1}\sigma}|\Omega(x-y)|\cdot|f(y)|\,dy.
\end{split}
\end{equation*}
For any $\mathcal{R}>0$ and $\Omega\in L^s(\mathbf{S}^{n-1})$, we first establish the following estimate, which will be often used in the sequel.
\begin{equation}\label{omega88}
\bigg(\int_{|x-y|<\mathcal{R}}|\Omega(x-y)|^s\,dy\bigg)^{1/s}\leq C\|\Omega\|_{L^s(\mathbf{S}^{n-1})}\mathcal{R}^{n/s}.
\end{equation}
To this end, using polar coordinates, we obtain
\begin{equation*}
\begin{split}
\bigg(\int_{|x-y|<\mathcal{R}}|\Omega(x-y)|^s\,dy\bigg)^{1/s}
&=\bigg(\int_{|y|<\mathcal{R}}|\Omega(y)|^{s}\,dy\bigg)^{1/s}\\
&=\bigg(\int_0^{\mathcal{R}}\int_{\mathbf{S}^{n-1}}|\Omega(y')|^{s}\varrho^{n-1}\,d\sigma(y')d\varrho\bigg)^{1/s}\\
&\leq C\|\Omega\|_{L^s(\mathbf{S}^{n-1})}\mathcal{R}^{n/s},
\end{split}
\end{equation*}
as desired. Here $y'=y/{|y|}$ for any $\mathbb R^n\ni y\neq0$. A combination of H\"{o}lder's inequality and \eqref{omega88} gives that for each fixed $j\geq1$,
\begin{equation*}
\begin{split}
&\int_{2^{-j}\sigma\leq|x-y|<2^{-j+1}\sigma}|\Omega(x-y)|\cdot|f(y)|\,dy\\
&\leq\bigg(\int_{|x-y|<2^{-j+1}\sigma}|\Omega(x-y)|^s\,dy\bigg)^{1/s}\bigg(\int_{|x-y|<2^{-j+1}\sigma}|f(y)|^{s'}\,dy\bigg)^{1/{s'}}\\
&\leq C\|\Omega\|_{L^s(\mathbf{S}^{n-1})}\big(2^{-j+1}\sigma\big)^{n/s}
\bigg(\int_{B(x,2^{-j+1}\sigma)}|f(y)|^{s'}\,dy\bigg)^{1/{s'}}\\
&=C\|\Omega\|_{L^s(\mathbf{S}^{n-1})}\big(2^{-j+1}\sigma\big)^{n/s}\cdot\big(2^{-j+1}\sigma\big)^{n/{s'}}\\
&\times\bigg(\frac{1}{|B(x,2^{-j+1}\sigma)|}\int_{B(x,2^{-j+1}\sigma)}|f(y)|^{s'}\,dy\bigg)^{1/{s'}}\\
&\leq C\|\Omega\|_{L^s(\mathbf{S}^{n-1})}\big(2^{-j+1}\sigma\big)^{n}M_{s'}f(x),
\end{split}
\end{equation*}
where the maximal operator $M_{s'}$ is given by
\begin{equation*}
M_{s'}f(x):=\big[M(|f|^{s'})(x)\big]^{1/{s'}},\quad x\in\mathbb R^n.
\end{equation*}
Note that $\alpha>0$. Hence
\begin{equation*}
\begin{split}
\mathrm{I}&\lesssim\sum_{j=1}^\infty\frac{1}{(2^{-j}\sigma)^{n-\alpha}}\cdot\big(2^{-j}\sigma\big)^{n}M_{s'}f(x)\\
&=\sum_{j=1}^\infty\frac{1}{(2^{-j}\sigma)^{-\alpha}}M_{s'}f(x)\lesssim\sigma^{\alpha}M_{s'}f(x).
\end{split}
\end{equation*}
On the other hand,
\begin{equation*}
\begin{split}
\mathrm{II}&=\sum_{j=1}^\infty\int_{2^{j-1}\sigma\leq|x-y|<2^j\sigma}\frac{|\Omega(x-y)|}{|x-y|^{n-\alpha}}|f(y)|\,dy\\
&\leq\sum_{j=1}^\infty\frac{1}{(2^{j-1}\sigma)^{n-\alpha}}\int_{2^{j-1}\sigma\leq|x-y|<2^j\sigma}|\Omega(x-y)|\cdot|f(y)|\,dy.
\end{split}
\end{equation*}
Since $\Omega\in L^s(\mathbf{S}^{n-1})$ with $s\geq p'$, then $\Omega\in L^{p'}(\mathbf{S}^{n-1})$, and
\begin{equation*}
\|\Omega\|_{L^{p'}(\mathbf{S}^{n-1})}\leq C\|\Omega\|_{L^s(\mathbf{S}^{n-1})}.
\end{equation*}
This estimate, along with H\"{o}lder's inequality and \eqref{omega88}, yields
\begin{equation*}
\begin{split}
&\int_{2^{j-1}\sigma\leq|x-y|<2^j\sigma}|\Omega(x-y)|\cdot|f(y)|\,dy\\
&\leq\bigg(\int_{|x-y|<2^j\sigma}|\Omega(x-y)|^{p'}\,dy\bigg)^{1/{p'}}\bigg(\int_{|x-y|<2^j\sigma}|f(y)|^p\,dy\bigg)^{1/p}\\
&\leq C\big(2^j\sigma\big)^{n/{p'}}\|f\|_{L^p},
\end{split}
\end{equation*}
for each integer $j\geq1$. Therefore,
\begin{equation*}
\begin{split}
\mathrm{II}&\lesssim\sum_{j=1}^\infty\frac{1}{(2^{j}\sigma)^{n-\alpha}}\cdot\big(2^j\sigma\big)^{n/{p'}}\|f\|_{L^p}\\
&=\sum_{j=1}^\infty\frac{1}{(2^{j}\sigma)^{n/p-\alpha}}\|f\|_{L^p}\lesssim \sigma^{\alpha-n/p}\|f\|_{L^p},
\end{split}
\end{equation*}
where the last series is convergent since $n/p-\alpha>0$. Summing up the above estimates for $\mathrm{I}$ and $\mathrm{II}$, for any $x\in\mathbb R^n$, we conclude that
\begin{equation}\label{main1}
\big|T_{\Omega,\alpha}f(x)\big|\lesssim\Big[\sigma^{\alpha}M_{s'}f(x)+\sigma^{\alpha-n/p}\|f\|_{L^p}\Big].
\end{equation}
We now choose $\sigma>0$ such that
\begin{equation*}
\sigma^{\alpha}M_{s'}f(x)=\sigma^{\alpha-n/p}\|f\|_{L^p}.
\end{equation*}
That is,
\begin{equation*}
\sigma^{n/p}=\frac{\|f\|_{L^p}}{M_{s'}f(x)}.
\end{equation*}
Putting this value of $\sigma$ back into \eqref{main1} and noting that $1-{\alpha p}/n=p/q$, we get
\begin{equation*}
\begin{split}
\big|T_{\Omega,\alpha}f(x)\big|\lesssim\sigma^{\alpha}M_{s'}f(x)&=\Big[\frac{\|f\|_{L^p}}{M_{s'}f(x)}\Big]^{{\alpha p}/n}\cdot M_{s'}f(x)\\
&=\big[M_{s'}f(x)\big]^{p/q}\cdot\big[\|f\|_{L^p}\big]^{1-p/q}.
\end{split}
\end{equation*}
Thus, \eqref{wanghua1} holds. The conclusion of Theorem \ref{thm1} then follows from \eqref{wanghua1} and the boundedness of $M$.
\begin{itemize}
  \item If $s'<p<n/{\alpha}$, then we have
\begin{equation*}
\begin{split}
\|T_{\Omega,\alpha}(f)\|_{L^q}&\lesssim\bigg(\int_{\mathbb R^n}\big|M_{s'}f(x)\big|^{p}dx\bigg)^{1/q} \cdot\big[\|f\|_{L^p}\big]^{1-p/q}\\
&=\bigg(\int_{\mathbb R^n}\big|M(|f|^{s'})(x)\big|^{p/{s'}}dx\bigg)^{1/q}\cdot\big[\|f\|_{L^p}\big]^{1-p/q}.
\end{split}
\end{equation*}
Thus, using the $L^{\frac{p}{s'}}$-boundedness of $M$($1<\frac{p}{s'}<\infty$), we obtain
\begin{equation*}
\begin{split}
\|T_{\Omega,\alpha}(f)\|_{L^q}&\lesssim\bigg(\int_{\mathbb R^n}|f(x)|^{s'\cdot(p/{s'})}dx\bigg)^{1/q}\cdot\big[\|f\|_{L^p}\big]^{1-p/q}\\
&=\big[\|f\|_{L^p}\big]^{p/q}\cdot\big[\|f\|_{L^p}\big]^{1-p/q}=\|f\|_{L^p}.
\end{split}
\end{equation*}
This proves $(ii)$.
  \item If $s'=p<n/{\alpha}$, then for any given $\lambda>0$, we have
\begin{equation*}
\begin{split}
&\lambda\cdot m\big(\big\{x\in\mathbb R^n:\big|T_{\Omega,\alpha}f(x)\big|>\lambda\big\}\big)^{1/q}\\
&=\lambda\cdot m\bigg(\bigg\{x\in\mathbb R^n:\big|M_{s'}f(x)\big|^{p/q}>\frac{\lambda}{C\|f\|^{1-p/q}_{L^p}}\bigg\}\bigg)^{1/q}\\
&=\lambda\cdot m\bigg(\bigg\{x\in\mathbb R^n:\big|M(|f|^{p})(x)\big|>\bigg(\frac{\lambda}{C\|f\|^{1-p/q}_{L^p}}\bigg)^q\bigg\}\bigg)^{1/q}.
\end{split}
\end{equation*}
By using the weak $(1,1)$ boundedness of $M$, the above expression is bounded by
\begin{equation*}
\begin{split}
&C\lambda\cdot\bigg[\bigg(\frac{\|f\|^{1-p/q}_{L^p}}{\lambda}\bigg)^q\int_{\mathbb R^n}|f(x)|^pdx\bigg]^{1/q}\\
&=C\big[\|f\|_{L^p}\big]^{1-p/q}\cdot\big[\|f\|_{L^p}\big]^{p/q}=C\|f\|_{L^p}.
\end{split}
\end{equation*}
This proves $(i)$ by taking the supremum over all $\lambda>0$.
\end{itemize}
The proof of Theorem \ref{thm1} is now complete.
\end{proof}
It is well known that the Hardy--Littlewood maximal operator $M$ is weak $(1,1)$ and strong $(p,p)$, $1<p\leq\infty$. The strong $(1,1)$ inequality is false. In fact, it never holds, as the following result shows (see \cite[p.36]{duoand}).
\begin{prop}
If $f\in L^1(\mathbb R^n)$ and is not identically $0$, then $Mf\notin L^1(\mathbb R^n)$.
\end{prop}
For any given ball $B$ in $\mathbb R^n$, we say that $f\log^{+}f\in L^1(B)$, if
\begin{equation*}
\int_{B}|f(x)|\log^+|f(x)|\,dx<\infty,
\end{equation*}
where $\log^+t=\max\{\log t,0\}$. Nevertheless, we do have the following result regarding local integrability of the maximal function.
\begin{prop}
Let $f$ be an integrable function supported in a ball $B\subset\mathbb R^n$. Then $Mf\in L^1(B)$ if and only if $f\log^{+}f\in L^1(B)$.
\end{prop}
This result is due to Stein \cite{stein2} (see also \cite[p.42]{duoand} and \cite[p.23]{stein}). In order to deal with the endpoint case, we need to introduce the space $L\log L(B)$, which is defined by
\begin{equation*}
L\log L(B):=\bigg\{f~\mathrm{is~supported~in}~B:\int_{B}|f(x)|\big(1+\log^+|f(x)|\big)\,dx<\infty\bigg\}.
\end{equation*}
The class $L\log L$ was originally studied by Stein \cite{stein2}. Unfortunately, the expression $\int_{B}|f(x)|\big(1+\log^+|f(x)|\big)\,dx$ does not define a norm. Given a ball $B\subset\mathbb R^n$, we define the following Luxemburg norm:
\begin{equation*}
\|f\|_{L\log L(B)}:=\inf\bigg\{\lambda>0:\int_{B}\frac{|f(x)|}{\lambda}\bigg(1+\log^{+}\Big(\frac{|f(x)|}{\lambda}\Big)\bigg)dx\leq1\bigg\}.
\end{equation*}
Then $L\log L(B)$ becomes a Banach function space (referred to as Orlicz space, which is a generalization of the $L^p$ space) with respect to the norm $\|\cdot\|_{L\log L(B)}$. We rely on the following result, as the author pointed out in \cite[p.42]{duoand}.
\begin{equation}\label{Mlog}
\|Mf\|_{L^1(B)}\leq C\|f\|_{L\log L(B)}.
\end{equation}
In general, for $1\leq p<\infty$, we now define the space $L^p\log L(B)$ as follows.
\begin{equation*}
L^p\log L(B):=\bigg\{f~\mathrm{is~supported~in}~B:\int_{B}|f(x)|^p\big(1+\log^+|f(x)|\big)\,dx<\infty\bigg\}.
\end{equation*}
Correspondingly, the Luxemburg norm of $f$ is defined by
\begin{equation*}
\|f\|_{L^p\log L(B)}:=\inf\bigg\{\lambda>0:\int_{B}\Big(\frac{|f(x)|}{\lambda}\Big)^p\bigg(1+\log^{+}\Big(\frac{|f(x)|}{\lambda}\Big)\bigg)dx\leq1\bigg\}.
\end{equation*}
By definition, it is obvious that for each given $p$, $L^p\log L(B)\subset L^p(B)$, and
\begin{equation}\label{lplogl}
\|f\|_{L^p(B)}\leq \|f\|_{L^p\log L(B)}.
\end{equation}
In addition, it is easy to verify that for any $1\leq p<\infty$,
\begin{equation}\label{pp}
\big\||f|^p\big\|_{L\log L(B)}\leq \|f\|^p_{L^p\log L(B)}.
\end{equation}
It follows immediately from \eqref{Mlog} and \eqref{pp} that
\begin{equation}\label{Mpp}
\|M(|f|^p)\|_{L^1(B)}\leq C\|f\|^p_{L^p\log L(B)}.
\end{equation}
Based on the above estimates, we are going to prove the following result for the endpoint case $p=s'$.
\begin{thm}\label{thm7}
Let $\Omega\in L^s(\mathbf{S}^{n-1})$ with $1<s\leq\infty$. Let $0<\alpha<n$, $1\leq p<q<\infty$ and $1/q=1/p-\alpha/n$. If $s'=p<n/{\alpha}$, then for any given ball $B\subset\mathbb R^n$, the inequality
\begin{equation*}
\|T_{\Omega,\alpha}(f)\|_{L^q(B)}\lesssim\|f\|_{L^p\log L(B)}
\end{equation*}
holds.
\end{thm}
\begin{proof}[Proof of Theorem $\ref{thm7}$]
Taking into consideration \eqref{wanghua1}, we get
\begin{equation*}
\begin{split}
\|T_{\Omega,\alpha}(f)\|_{L^q(B)}&\lesssim\bigg(\int_{B}\big|M_{s'}f(x)\big|^{p}dx\bigg)^{1/q}\cdot\big[\|f\|_{L^p(B)}\big]^{1-p/q}\\
&=\bigg(\int_{B}\big|M(|f|^{p})(x)\big|\,dx\bigg)^{1/q}\cdot\big[\|f\|_{L^p(B)}\big]^{1-p/q}.
\end{split}
\end{equation*}
Then we can easily conclude from \eqref{Mpp} and \eqref{lplogl} that
\begin{equation*}
\|T_{\Omega,\alpha}(f)\|_{L^q(B)}\leq C\big[\|f\|_{L^p\log L(B)}\big]^{p/q}\cdot\big[\|f\|_{L^p\log L(B)}\big]^{1-p/q}=C\|f\|_{L^p\log L(B)}.
\end{equation*}
The proof is now complete.
\end{proof}
Before stating our next theorem, let us now introduce some notations.
\begin{defin}[\cite{ding2}]\label{defd}
We say that $\Omega$ satisfies the $L^s$-Dini smoothness condition $\mathfrak{D}_{s}$, if $\Omega\in L^s(\mathbf{S}^{n-1})$ is homogeneous of degree zero on $\mathbb R^n$ with $1\leq s<\infty$, and
\begin{equation*}
\int_0^1\frac{\omega_s(\delta)}{\delta}\,d\delta<\infty,
\end{equation*}
where $\omega_s(\delta)$ denotes the integral modulus of continuity of order $s$, which is defined by
\begin{equation*}
\omega_s(\delta):=\sup_{|\rho|<\delta}\bigg(\int_{\mathbf{S}^{n-1}}\big|\Omega(\rho x')-\Omega(x')\big|^sd\sigma(x')\bigg)^{1/s}
\end{equation*}
and $\rho$ is a rotation in $\mathbb R^n$ and $|\rho|:=\|\rho-I\|=\sup_{x'\in\mathbf{S}^{n-1}}\big|\rho x'-x'\big|$. We also say that $\Omega$ satisfies the $L^{\infty}$-Dini smoothness condition $\mathfrak{D}_{\infty}$, if $\Omega\in L^{\infty}(\mathbf{S}^{n-1})$ is homogeneous of degree zero on $\mathbb R^n$, and
\begin{equation*}
\int_0^1\frac{\omega_{\infty}(\delta)}{\delta}\,d\delta<\infty,
\end{equation*}
where $\omega_{\infty}(\delta)$ is defined by
\begin{equation*}
\omega_{\infty}(\delta):=\sup_{|\rho|<\delta,x'\in \mathbf{S}^{n-1}}\big|\Omega(\rho x')-\Omega(x')\big|,
\end{equation*}
where $x'=x/{|x|}$ for any $\mathbb R^n\ni x\neq0$.
\end{defin}
We need the following estimate which can be found in \cite[Lemma 1]{ding2}.
\begin{lem}[\cite{ding2}]\label{lem1}
Suppose that $0<\alpha<n$ and $\Omega$ satisfies the $L^s$-Dini smoothness condition $\mathfrak{D}_{s}$ with $1<s\le\infty$. If there exists a constant $0<\tau<1/2$ such that if $|x|<\tau R$, then we have
\begin{equation*}
\bigg(\int_{R\leq|z|<2R}\bigg|\frac{\Omega(z-x)}{|z-x|^{n-\alpha}}-\frac{\Omega(z)}{|z|^{n-\alpha}}\bigg|^sdz\bigg)^{1/s}
\leq C\cdot R^{n/s-(n-\alpha)}\bigg(\frac{|x|}{R}+\int_{|x|/{2R}}^{|x|/R}\frac{\omega_{s}(\delta)}{\delta}\,d\delta\bigg),
\end{equation*}
where the constant $C>0$ is independent of $R$ and $x$.
\end{lem}
For the critical case $p=n/{\alpha}$, we give the following inequality.
\begin{thm}\label{thm9}
Suppose that $\Omega$ satisfies the $L^s$-Dini smoothness condition $\mathfrak{D}_{s}$ with $s\geq n/{(n-\alpha)}$. Let $0<\alpha<n$ and $p=n/{\alpha}$. Then the inequality
\begin{equation*}
\|T_{\Omega,\alpha}(f)\|_{\mathrm{BMO}}\lesssim\|f\|_{L^{p}}
\end{equation*}
holds.
\end{thm}
\begin{proof}[Proof of Theorem $\ref{thm9}$]
Actually, the case $s>n/{(n-\alpha)}$ in this theorem was proved by Ding and Lu (see the unweighted version of Theorem 1 in \cite{ding2}). We only give the proof of the case $s=n/{(n-\alpha)}$ ($s=p'$ when $p=n/{\alpha}$). For any ball $B=B(x_0,r)$ with $(x_0,r)\in\mathbb R^n\times(0,\infty)$ and $f\in L^{p}(\mathbb R^n)$, we are going to give the estimation of the following expression
\begin{equation}\label{mainw}
\frac{1}{m(B)}\int_{B}\big|T_{\Omega,\alpha}f(x)-(T_{\Omega,\alpha}f)_{B}\big|\,dx.
\end{equation}
For this purpose, we decompose $f$ as $f=f_1+f_2$, where $f_1=f\cdot\chi_{4B}$, $f_2=f\cdot\chi_{(4B)^{\complement}}$, $4B=B(x_0,4r)$. By using the linearity of $T_{\Omega,\alpha}$, the expression \eqref{mainw} will be divided into two parts. That is,
\begin{equation*}
\begin{split}
&\frac{1}{m(B)}\int_{B}\big|T_{\Omega,\alpha}f(x)-(T_{\Omega,\alpha}f)_{B}\big|\,dx\\
&\leq \frac{1}{m(B)}\int_{B}\big|T_{\Omega,\alpha}f_1(x)-(T_{\Omega,\alpha}f_1)_{B}\big|\,dx
+\frac{1}{m(B)}\int_{B}\big|T_{\Omega,\alpha}f_2(x)-(T_{\Omega,\alpha}f_2)_{B}\big|\,dx\\
&:=\mathrm{I_1+I_2}.
\end{split}
\end{equation*}
Let us first consider the term $\mathrm{I_1}$. When $x\in B$ and $y\in 4B$, we have $|x-y|<5r$. An application of Fubini's theorem yields
\begin{equation*}
\begin{split}
\mathrm{I_1}&\leq\frac{2}{m(B)}\int_{B}|T_{\Omega,\alpha}f_1(x)|\,dx\\
&\leq\frac{2}{m(B)}\int_{B}\bigg(\int_{4B}\frac{|\Omega(x-y)|}{|x-y|^{n-\alpha}}|f(y)|\,dy\bigg)dx\\
&\leq\frac{2}{m(B)}\int_{4B}|f(y)|\bigg(\int_{|x-y|<5r}\frac{|\Omega(x-y)|}{|x-y|^{n-\alpha}}\,dx\bigg)dy.
\end{split}
\end{equation*}
Since $\Omega\in L^s(\mathbf{S}^{n-1})$, by H\"{o}lder's inequality, we know that $\Omega\in L^1(\mathbf{S}^{n-1})$, and
\begin{equation*}
\|\Omega\|_{L^1(\mathbf{S}^{n-1})}\leq C\|\Omega\|_{L^s(\mathbf{S}^{n-1})}.
\end{equation*}
Note that $\alpha>0$. Using polar coordinates, we can deduce that
\begin{equation}\label{wang1}
\begin{split}
&\int_{|x-y|<5r}\frac{|\Omega(x-y)|}{|x-y|^{n-\alpha}}\,dx\\
&=\int_{0}^{5r}\int_{\mathbf{S}^{n-1}}\frac{|\Omega(x')|}{\varrho^{n-\alpha}}\varrho^{n-1}d\sigma(x')d\varrho\\
&=C\cdot r^{\alpha}\|\Omega\|_{L^1(\mathbf{S}^{n-1})}\leq Cm(B)^{\alpha/n}\|\Omega\|_{L^s(\mathbf{S}^{n-1})}.
\end{split}
\end{equation}
Applying the H\"older inequality, we obtain
\begin{equation}\label{wang2}
\begin{split}
\int_{4B}|f(y)|\,dy
&\leq\bigg(\int_{4B}|f(y)|^p\,dy\bigg)^{1/p}m(4B)^{1/{p'}}.
\end{split}
\end{equation}
Notice that
\begin{equation}\label{cal3}
1-\frac{\alpha}{\,n\,}=1-\frac{\,1\,}{p}=\frac{1}{p'}.
\end{equation}
Hence, from \eqref{wang1}, \eqref{wang2} and \eqref{cal3}, it follows that
\begin{equation*}
\begin{split}
\mathrm{I_1}&\leq C\cdot\|\Omega\|_{L^s(\mathbf{S}^{n-1})}\frac{1}{m(B)^{1-\alpha/n}}\int_{4B}|f(y)|\,dy\\
&\leq C\cdot\|\Omega\|_{L^s(\mathbf{S}^{n-1})}\frac{1}{m(B)^{1-\alpha/n}}\cdot m(4B)^{1/{p'}}
\bigg(\int_{4B}|f(y)|^p\,dy\bigg)^{1/p}\\
&\leq C\cdot\|\Omega\|_{L^s(\mathbf{S}^{n-1})}\|f\|_{L^{p}},
\end{split}
\end{equation*}
as desired. Let us now turn to the estimate of the second term $\mathrm{I_2}$. It is easy to see that
\begin{equation*}
\begin{split}
\mathrm{I_2}&=\frac{1}{m(B)}\int_{B}\big|T_{\Omega,\alpha}f_2(x)-(T_{\Omega,\alpha}f_2)_{B}\big|\,dx\\
&\leq\frac{1}{m(B)}\int_{B}\bigg\{\frac{1}{m(B)}\int_{B}\big|T_{\Omega,\alpha}f_2(x)-T_{\Omega,\alpha}f_2(y)\big|\,dy\bigg\}dx\\
&\leq\frac{1}{m(B)}\int_{B}\bigg\{\frac{1}{m(B)}\int_{B}\bigg(\sum_{k=2}^\infty\int_{2^{k+1}B\backslash 2^kB}
\Big|\frac{\Omega(x-z)}{|x-z|^{n-\alpha}}-\frac{\Omega(y-z)}{|y-z|^{n-\alpha}}\Big||f(z)|dz\bigg)dy\bigg\}dx.
\end{split}
\end{equation*}
Since
\begin{equation*}
\begin{split}
\Big|\frac{\Omega(x-z)}{|x-z|^{n-\alpha}}-\frac{\Omega(y-z)}{|y-z|^{n-\alpha}}\Big|
&\leq\Big|\frac{\Omega(x-z)}{|x-z|^{n-\alpha}}-\frac{\Omega(x_0-z)}{|x_0-z|^{n-\alpha}}\Big|\\
&+\Big|\frac{\Omega(x_0-z)}{|x_0-z|^{n-\alpha}}-\frac{\Omega(y-z)}{|y-z|^{n-\alpha}}\Big|,
\end{split}
\end{equation*}
by using H\"{o}lder's inequality and Minkowski's inequality($s'=p$), then we have that for each integer $k\geq2$,
\begin{equation*}
\begin{split}
&\int_{2^{k+1}B\backslash 2^kB}
\Big|\frac{\Omega(x-z)}{|x-z|^{n-\alpha}}-\frac{\Omega(y-z)}{|y-z|^{n-\alpha}}\Big||f(z)|dz\\
&\leq\bigg(\int_{2^{k+1}B\backslash 2^kB}\Big|\frac{\Omega(x-z)}{|x-z|^{n-\alpha}}-\frac{\Omega(y-z)}{|y-z|^{n-\alpha}}\Big|^sdz\bigg)^{1/s}
\bigg(\int_{2^{k+1}B\backslash 2^kB}|f(z)|^{s'}dz\bigg)^{1/{s'}}\\
&\leq\bigg(\int_{2^{k+1}B\backslash 2^kB}\Big|\frac{\Omega(x-z)}{|x-z|^{n-\alpha}}-\frac{\Omega(x_0-z)}{|x_0-z|^{n-\alpha}}\Big|^sdz\bigg)^{1/s}
\|f\|_{L^{p}}\\
&+\bigg(\int_{2^{k+1}B\backslash 2^kB}\Big|\frac{\Omega(x_0-z)}{|x_0-z|^{n-\alpha}}-\frac{\Omega(y-z)}{|y-z|^{n-\alpha}}\Big|^sdz\bigg)^{1/s}
\|f\|_{L^{p}}.
\end{split}
\end{equation*}
For any given $x\in B$ and $z\in 2^{k+1}B\backslash 2^kB$ with $k\geq2$, a trivial computation shows that
\begin{equation*}
|x-x_0|<\frac{\,1\,}{4}|z-x_0|.
\end{equation*}
According to Lemma \ref{lem1}(if we take $R=2^kr$), we get
\begin{equation*}
\begin{split}
&\bigg(\int_{2^{k+1}B\backslash 2^kB}\bigg|\frac{\Omega(x-z)}{|x-z|^{n-\alpha}}-\frac{\Omega(x_0-z)}{|x_0-z|^{n-\alpha}}\bigg|^sdz\bigg)^{1/s}\\
&=\bigg(\int_{2^kr\leq|z-x_0|<2^{k+1}r}\bigg|\frac{\Omega(z-x_0-(x-x_0))}{|z-x_0-(x-x_0)|^{n-\alpha}}
-\frac{\Omega(z-x_0)}{|z-x_0|^{n-\alpha}}\bigg|^sdz\bigg)^{1/s}\\
&\leq C\big(2^kr\big)^{n/s-(n-\alpha)}
\bigg[\frac{|x-x_0|}{2^kr}+\int_{\frac{|x-x_0|}{2^{k+1}r}}^{\frac{|x-x_0|}{2^kr}}\frac{\omega_s(\delta)}{\delta}d\delta\bigg]\\
&\leq C\bigg[\frac{1}{2^k}+\int_{\frac{|x-x_0|}{2^{k+1}r}}^{\frac{|x-x_0|}{2^kr}}\frac{\omega_s(\delta)}{\delta}d\delta\bigg],
\end{split}
\end{equation*}
where in the last step we have used the fact that $\alpha=n/p$. Similarly, for any given $y\in B$, we can also obtain
\begin{equation*}
\begin{split}
&\bigg(\int_{2^{k+1}B\backslash 2^kB}\Big|\frac{\Omega(x_0-z)}{|x_0-z|^{n-\alpha}}-\frac{\Omega(y-z)}{|y-z|^{n-\alpha}}\Big|^sdz\bigg)^{1/s}\\
&\leq C\big(2^kr\big)^{n/s-(n-\alpha)}
\bigg[\frac{|y-x_0|}{2^kr}+\int_{\frac{|y-x_0|}{2^{k+1}r}}^{\frac{|y-x_0|}{2^kr}}\frac{\omega_s(\delta)}{\delta}d\delta\bigg]\\
&\leq C\bigg[\frac{1}{2^k}+\int_{\frac{|y-x_0|}{2^{k+1}r}}^{\frac{|y-x_0|}{2^kr}}\frac{\omega_s(\delta)}{\delta}d\delta\bigg].
\end{split}
\end{equation*}
Consequently, for any $x\in B$ and $y\in B$, we deduce that
\begin{equation*}
\begin{split}
&\sum_{k=2}^\infty\int_{2^{k+1}B\backslash 2^kB}
\Big|\frac{\Omega(x-z)}{|x-z|^{n-\alpha}}-\frac{\Omega(y-z)}{|y-z|^{n-\alpha}}\Big||f(z)|\,dz\\
&\leq C\|f\|_{L^{p}}\sum_{k=2}^\infty
\bigg[\frac{1}{2^{k-1}}+\int_{\frac{|x-x_0|}{2^{k+1}r}}^{\frac{|x-x_0|}{2^kr}}\frac{\omega_s(\delta)}{\delta}d\delta
+\int_{\frac{|y-x_0|}{2^{k+1}r}}^{\frac{|y-x_0|}{2^kr}}\frac{\omega_s(\delta)}{\delta}d\delta\bigg]\\
&\leq C\|f\|_{L^{p}}\bigg[1+2\int_0^1\frac{\omega_s(\delta)}{\delta}d\delta\bigg].
\end{split}
\end{equation*}
Therefore, we obtain
\begin{equation*}
\begin{split}
\mathrm{I_2}&\leq\frac{1}{m(B)}\int_{B}\bigg\{\frac{1}{m(B)}\int_{B}\bigg(\sum_{k=2}^\infty\int_{2^{k+1}B\backslash 2^kB}
\Big|\frac{\Omega(x-z)}{|x-z|^{n-\alpha}}-\frac{\Omega(y-z)}{|y-z|^{n-\alpha}}\Big||f(z)|\,dz\bigg)dy\bigg\}dx\\
&\leq C\bigg[1+2\int_0^1\frac{\omega_s(\delta)}{\delta}d\delta\bigg]\|f\|_{L^{p}}.
\end{split}
\end{equation*}
Combining the above estimates for $\mathrm{I_1}$ and $\mathrm{I_2}$ and taking the supremum over all balls $B\subset\mathbb R^n$, we conclude the proof of Theorem \ref{thm9}.
\end{proof}
For the critical case $p=n/{\alpha}$, we also show that the fractional maximal operator $M_{\Omega,\alpha}$ is bounded from $L^p(\mathbb R^n)$ to $L^{\infty}(\mathbb R^n)$.
\begin{thm}\label{thm10}
Suppose that $\Omega\in L^s(\mathbf{S}^{n-1})$ with $s\geq n/{(n-\alpha)}$. Let $0<\alpha<n$ and $p=n/{\alpha}$. Then the inequality
\begin{equation*}
\|M_{\Omega,\alpha}(f)\|_{L^{\infty}}\lesssim\|f\|_{L^{p}}
\end{equation*}
holds.
\end{thm}
\begin{proof}[Proof of Theorem $\ref{thm10}$]
Given a ball $B(x,r)$ with center $x\in\mathbb R^n$ and radius $r\in(0,\infty)$, by using H\"{o}lder's inequality and \eqref{omega88}, we get
\begin{equation*}
\begin{split}
&\frac{1}{m(B(x,r))^{1-\alpha/n}}\int_{|x-y|<r}|\Omega(x-y)|\cdot|f(y)|\,dy\\
&\leq\frac{1}{m(B(x,r))^{1/{p'}}}\bigg(\int_{|x-y|<r}|\Omega(x-y)|^sdy\bigg)^{1/s}\bigg(\int_{|x-y|<r}|f(y)|^{s'}dy\bigg)^{1/{s'}}\\
&\leq C\|\Omega\|_{L^s(\mathbf{S}^{n-1})}\cdot\frac{m(B(x,r))^{1/s}}{m(B(x,r))^{1/{p'}}}\bigg(\int_{|x-y|<r}|f(y)|^{s'}dy\bigg)^{1/{s'}}.
\end{split}
\end{equation*}
Two cases are considered below.

\textbf{Case 1:} $s=n/{(n-\alpha)}$. In this case, one has $p'=s$ when $p=n/{\alpha}$, then the result holds trivially.

\textbf{Case 2:} $s>n/{(n-\alpha)}$. A direct computation shows that $p>s'$ and
\begin{equation*}
\begin{split}
\frac{1}{s'(p/{s'})'}=\frac{1}{s'}\cdot\frac{p-s'}{p}=\frac{1}{s'}-\frac{\,1\,}{p}=\frac{1}{p'}-\frac{\,1\,}{s}.
\end{split}
\end{equation*}
This, combined with H\"{o}lder's inequality, yields
\begin{equation*}
\begin{split}
\bigg(\int_{|x-y|<r}|f(y)|^{s'}dy\bigg)^{1/{s'}}
&\leq\bigg(\int_{B(x,r)}|f(y)|^{p}\,dy\bigg)^{1/p}\big(m(B(x,r))\big)^{\frac{1}{s'(p/{s'})'}}\\
&=\bigg(\int_{B(x,r)}|f(y)|^{p}\,dy\bigg)^{1/p}\big(m(B(x,r))\big)^{1/{p'}-1/s}.
\end{split}
\end{equation*}
Hence
\begin{equation*}
\begin{split}
&\frac{1}{m(B(x,r))^{1-\alpha/n}}\int_{|x-y|<r}|\Omega(x-y)|\cdot|f(y)|\,dy\\
&\leq C\|\Omega\|_{L^s(\mathbf{S}^{n-1})}\bigg(\int_{B(x,r)}|f(y)|^p\,dy\bigg)^{1/p}
\leq C\|\Omega\|_{L^s(\mathbf{S}^{n-1})}\|f\|_{L^p}.
\end{split}
\end{equation*}
We are done.
\end{proof}
\begin{rem}
For $0<\alpha<n$, we can show that there exist positive constants $C_1,C_2>0$ independent of $f$ such that
\begin{equation*}
C_1M_{\alpha}f(x)\leq M^{\#}(I_{\alpha}f)(x)\leq C_2M_{\alpha}f(x).
\end{equation*}
Here $M^{\#}$ is the sharp maximal function of Fefferman and Stein given by
\begin{equation*}
M^{\#}f(x):=\sup_{B\ni x}\frac{1}{m(B)}\int_B|f(y)-f_B|\,dy,
\end{equation*}
where the supremum is taken over all balls $B$ containing $x$. These results were proved by Adams \cite{adams} (see also \cite{adams1}). Thus, for any locally integrable function $f$, we have
\begin{equation*}
M_{\alpha}(f)\in L^{\infty}(\mathbb R^n)\Longleftrightarrow I_{\alpha}(f)\in \mathrm{BMO}(\mathbb R^n).
\end{equation*}
We claim that if $\Omega$ satisfies the assumptions of Theorem \ref{thm9}, then there exist positive constants $C_1$ and $C_2$ independent of $f$ such that
\begin{equation*}
C_1M_{\Omega,\alpha}f(x)\leq M^{\#}(T_{\Omega,\alpha}f)(x)\leq C_2M_{\Omega,\alpha}f(x).\quad (\dag)
\end{equation*}
By using similar arguments as in the proof of Theorem \ref{thm9}, we can prove $(\dag)$ as well. The details are omitted here. Consequently, when $\Omega\in\mathfrak{D}_{s}$ with $s\geq n/{(n-\alpha)}$, we have
\begin{equation*}
M_{\Omega,\alpha}(f)\in L^{\infty}(\mathbb R^n)\Longleftrightarrow T_{\Omega,\alpha}(f)\in \mathrm{BMO}(\mathbb R^n).\quad (\ddag)
\end{equation*}
Hence, if one has a stronger assumption that $\Omega\in\mathfrak{D}_{s}$, then the proof of Theorem \ref{thm10} follows immediately from $(\ddag)$.
\end{rem}

\subsection{Boundedness of $T_{\Omega,\alpha}$ and $M_{\Omega,\alpha}$ on Morrey spaces}
\begin{thm}\label{thm2}
Let $\Omega\in L^s(\mathbf{S}^{n-1})$ with $1<s\leq\infty$. Let $0<\alpha<n$, $1\leq p<q<\infty$, $0<\kappa<1-{(\alpha p)}/n$ and $1/q=1/p-\alpha/{n(1-\kappa)}$.
$(i)$ If $s'=p$, then for every $p<n/{\alpha}$, the inequality
\begin{equation*}
\|T_{\Omega,\alpha}(f)\|_{WL^{q,\kappa}}\lesssim\|f\|_{L^{p,\kappa}}
\end{equation*}
holds.

$(ii)$ If $s'<p$, then for every $p\in(s',n/{\alpha})$, the inequality
\begin{equation*}
\|T_{\Omega,\alpha}(f)\|_{L^{q,\kappa}}\lesssim\|f\|_{L^{p,\kappa}}
\end{equation*}
holds.
\end{thm}
\begin{proof}[Proof of Theorem \ref{thm2}]
The idea of the proof is due to Adams \cite{adams,adams1}. We will prove the following result.
\begin{equation}\label{wanghua2}
\begin{split}
\big|T_{\Omega,\alpha}f(x)\big|\lesssim\big[\|f\|_{L^{p,\kappa}}\big]^{1-p/q}\cdot\big[M_{s'}f(x)\big]^{p/q},\quad x\in\mathbb R^n.
\end{split}
\end{equation}
Let $f\in L^{p,\kappa}(\mathbb R^n)$ with $s'\leq p<n/{\alpha}$ and $0<\kappa<1-{(\alpha p)}/n$. In order to prove \eqref{wanghua2}, for any $x\in\mathbb R^n$, we write
\begin{equation*}
\begin{split}
\big|T_{\Omega,\alpha}f(x)\big|&\leq \int_{\mathbb R^n}\frac{|\Omega(x-y)|}{|x-y|^{n-\alpha}}|f(y)|\,dy\\
&=\int_{|x-y|<\sigma}\frac{|\Omega(x-y)|}{|x-y|^{n-\alpha}}|f(y)|\,dy
+\int_{|x-y|\geq\sigma}\frac{|\Omega(x-y)|}{|x-y|^{n-\alpha}}|f(y)|\,dy\\
&:=\mathrm{III+IV},
\end{split}
\end{equation*}
where $\sigma$ is a positive number to be chosen later. Invoking the same argument involving $\mathrm{I}$ as in Theorem \ref{thm1}, we then have
\begin{equation*}
\begin{split}
\mathrm{III}&=\sum_{j=1}^\infty\int_{2^{-j}\sigma\leq|x-y|<2^{-j+1}\sigma}\frac{|\Omega(x-y)|}{|x-y|^{n-\alpha}}|f(y)|\,dy\\
&\lesssim\sum_{j=1}^\infty\frac{1}{(2^{-j}\sigma)^{n-\alpha}}\cdot\big(2^{-j}\sigma\big)^{n}M_{s'}f(x)\lesssim\sigma^{\alpha}M_{s'}f(x).
\end{split}
\end{equation*}
We now give the estimate for the other term $\mathrm{IV}$. It is easy to see that
\begin{equation*}
\begin{split}
\mathrm{IV}&=\sum_{j=1}^\infty\int_{2^{j-1}\sigma\leq|x-y|<2^j\sigma}\frac{|\Omega(x-y)|}{|x-y|^{n-\alpha}}|f(y)|\,dy\\
&\leq\sum_{j=1}^\infty\frac{1}{(2^{j-1}\sigma)^{n-\alpha}}\int_{2^{j-1}\sigma\leq|x-y|<2^j\sigma}|\Omega(x-y)|\cdot|f(y)|\,dy.
\end{split}
\end{equation*}
Since $\Omega\in L^s(\mathbf{S}^{n-1})$ with $s\geq p'$, then $\Omega\in L^{p'}(\mathbf{S}^{n-1})$. Using H\"{o}lder's inequality and the estimate \eqref{omega88}, we get
\begin{equation*}
\begin{split}
&\int_{2^{j-1}\sigma\leq|x-y|<2^j\sigma}|\Omega(x-y)|\cdot|f(y)|\,dy\\
&\leq\bigg(\int_{|x-y|<2^j\sigma}|\Omega(x-y)|^{p'}\,dy\bigg)^{1/{p'}}\bigg(\int_{|x-y|<2^j\sigma}|f(y)|^p\,dy\bigg)^{1/p}\\
&\leq C\big(2^j\sigma\big)^{n/{p'}}\cdot\big(2^j\sigma\big)^{{\kappa n}/p}\|f\|_{L^{p,\kappa}},
\end{split}
\end{equation*}
for each integer $j\geq1$. Therefore,
\begin{equation*}
\begin{split}
\mathrm{IV}&\lesssim\sum_{j=1}^\infty\frac{1}{(2^{j}\sigma)^{n-\alpha}}\cdot\big(2^j\sigma\big)^{n/{p'}}
\cdot\big(2^j\sigma\big)^{{\kappa n}/p}\|f\|_{L^{p,\kappa}}\\
&=\sum_{j=1}^\infty\frac{1}{(2^{j}\sigma)^{n/p-{\kappa n}/p-\alpha}}\|f\|_{L^{p,\kappa}}
\lesssim\sigma^{\alpha-{(1-\kappa)n}/p}\|f\|_{L^{p,\kappa}},
\end{split}
\end{equation*}
where the last series is convergent since ${(1-\kappa)}/p-\alpha/n>0$. Collecting all these estimates for $\mathrm{III}$ and $\mathrm{IV}$, for any $x\in\mathbb R^n$, we conclude that
\begin{equation}\label{main2}
\big|T_{\Omega,\alpha}f(x)\big|\lesssim\Big[\sigma^{\alpha}M_{s'}f(x)+\sigma^{\alpha-{(1-\kappa)n}/p}\|f\|_{L^{p,\kappa}}\Big].
\end{equation}
We now choose $\sigma>0$ such that
\begin{equation*}
\sigma^{\alpha}M_{s'}f(x)=\sigma^{\alpha-{(1-\kappa)n}/p}\|f\|_{L^{p,\kappa}}.
\end{equation*}
That is,
\begin{equation*}
\sigma^{{(1-\kappa)n}/p}=\frac{\|f\|_{L^{p,\kappa}}}{M_{s'}f(x)}.
\end{equation*}
Putting this value of $\sigma$ back into \eqref{main2} and noting that
\begin{equation*}
1-{\alpha p}/{(1-\kappa)n}=p\cdot\big[1/p-\alpha/{(1-\kappa)n}\big]=p/q,
\end{equation*}
we obtain
\begin{equation*}
\begin{split}
\big|T_{\Omega,\alpha}f(x)\big|\lesssim\sigma^{\alpha}M_{s'}f(x)&=\Big[\frac{\|f\|_{L^{p,\kappa}}}{M_{s'}f(x)}\Big]^{{\alpha p}/{(1-\kappa)n}}\cdot M_{s'}f(x)\\
&=\big[M_{s'}f(x)\big]^{p/q}\cdot\big[\|f\|_{L^{p,\kappa}}\big]^{1-p/q}.
\end{split}
\end{equation*}
Thus, \eqref{wanghua2} holds. We also need the following lemma in \cite{adams1} (see also \cite{komori} for the weighted case).
\begin{lem}\label{MM}
Let $1<p<\infty$ and $0<\kappa<1$. Then the Hardy--Littlewood maximal operator $M$ is bounded on $L^{p,\kappa}(\mathbb R^n)$, and bounded from $L^{1,\kappa}(\mathbb R^n)$ to $WL^{1,\kappa}(\mathbb R^n)$.
\end{lem}
The conclusion of Theorem \ref{thm2} then follows from \eqref{wanghua2} and the boundedness of $M$ on Morrey spaces.
\begin{itemize}
  \item If $s'<p<n/{\alpha}$, we first observe that
\begin{equation}\label{fcase}
\begin{split}
\big\||f|^{s'}\big\|_{L^{p/{s'},\kappa}}
=&\sup_{B\subseteq\mathbb R^n}\bigg(\frac{1}{m(B)^{\kappa}}\int_B|f(x)|^{s'\cdot(p/{s'})}\,dx\bigg)^{{s'}/p}\\
=&\bigg[\sup_{B\subseteq\mathbb R^n}\bigg(\frac{1}{m(B)^{\kappa}}\int_B|f(x)|^p\,dx\bigg)^{1/p}\bigg]^{s'}\\
=&\big[\|f\|_{L^{p,\kappa}}\big]^{s'}.
\end{split}
\end{equation}
An application of \eqref{wanghua2} yields
\begin{equation*}
\begin{split}
\|T_{\Omega,\alpha}(f)\|_{L^{q,\kappa}}
&\lesssim\sup_{B\subseteq\mathbb R^n}\bigg(\frac{1}{m(B)^{\kappa}}\int_{B}\big|M_{s'}f(x)\big|^{p}dx\bigg)^{1/q} \cdot\big[\|f\|_{L^{p,\kappa}}\big]^{1-p/q}\\
&=\sup_{B\subseteq\mathbb R^n}\bigg(\frac{1}{m(B)^{\kappa}}\int_{B}\big|M(|f|^{s'})(x)\big|^{p/{s'}}dx\bigg)^{1/{q}}
\cdot\big[\|f\|_{L^{p,\kappa}}\big]^{1-p/q}.
\end{split}
\end{equation*}
Moreover, in view of \eqref{fcase} and Lemma \ref{MM}, we have
\begin{equation*}
\begin{split}
\|T_{\Omega,\alpha}(f)\|_{L^{q,\kappa}}&\lesssim
\big[\big\||f|^{s'}\big\|_{L^{p/{s'},\kappa}}\big]^{p/{(s'q)}}
\cdot\big[\|f\|_{L^{p,\kappa}}\big]^{1-p/q}\\
&=\big[\|f\|_{L^{p,\kappa}}\big]^{p/q}\cdot\big[\|f\|_{L^{p,\kappa}}\big]^{1-p/q}=\|f\|_{L^{p,\kappa}}.
\end{split}
\end{equation*}
This shows $(ii)$.
  \item If $s'=p<n/{\alpha}$, we observe that
\begin{equation}\label{scase}
\begin{split}
\big\||f|^{p}\big\|_{L^{1,\kappa}}
=&\sup_{B'\subseteq\mathbb R^n}\bigg(\frac{1}{m(B')^{\kappa}}\int_{B'}|f(x)|^{p}\,dx\bigg)\\
=&\bigg[\sup_{B'\subseteq\mathbb R^n}\bigg(\frac{1}{m(B')^{\kappa}}\int_{B'}|f(x)|^p\,dx\bigg)^{1/p}\bigg]^{p}\\
=&\big[\|f\|_{L^{p,\kappa}}\big]^{p}.
\end{split}
\end{equation}
By using the estimate \eqref{wanghua2}, for any given ball $B$ in $\mathbb R^n$, we then have
\begin{equation*}
\begin{split}
&\frac{1}{m(B)^{\kappa/q}}\sup_{\lambda>0}\lambda\cdot m\big(\big\{x\in B:\big|T_{\Omega,\alpha}f(x)\big|>\lambda\big\}\big)^{1/q}\\
&\leq \frac{1}{m(B)^{\kappa/q}}\sup_{\lambda>0}\lambda\cdot m\bigg(\bigg\{x\in B:\big|M_{s'}f(x)\big|^{p/q}>\frac{\lambda}{C\|f\|^{1-p/q}_{L^{p,\kappa}}}\bigg\}\bigg)^{1/q}\\
&=\bigg[\frac{1}{m(B)^{\kappa}}\sup_{\lambda>0}\lambda^q\cdot m\bigg(\bigg\{x\in B:\big|M(|f|^{p})(x)\big|>\bigg(\frac{\lambda}{C\|f\|^{1-p/q}_{L^{p,\kappa}}}\bigg)^q\bigg\}\bigg)\bigg]^{1/q}.
\end{split}
\end{equation*}
In view of \eqref{scase} and Lemma \ref{MM}, the above expression is bounded by
\begin{equation*}
\begin{split}
&C\Big[\big(\|f\|^{1-p/q}_{L^{p,\kappa}}\big)^q\cdot\big\||f|^{p}\big\|_{L^{1,\kappa}}\Big]^{1/q}\\
&=C\big[\|f\|_{L^{p,\kappa}}\big]^{1-p/q}\cdot\big[\|f\|_{L^{p,\kappa}}\big]^{p/q}=C\|f\|_{L^{p,\kappa}}.
\end{split}
\end{equation*}
This shows $(i)$ by taking the supremum over all balls $B$ in $\mathbb R^n$.
\end{itemize}
The proof of Theorem \ref{thm2} is now complete.
\end{proof}

Next we deal with the endpoint case within the framework of the Morrey space. For $1\leq p<\infty$ and $0<\kappa<1$, we now define the space $(L\log L)^{p,\kappa}(\mathbb R^n)$ as the set of all locally integrable functions $f$ on $\mathbb R^n$ such that
\begin{equation*}
\begin{split}
\|f\|_{(L\log L)^{p,\kappa}}:=&\sup_{B}\frac{1}{m(B)^{\kappa/p}}\big\|f\big\|_{L^p\log L(B)}<\infty.
\end{split}
\end{equation*}
When $p=1$, we write $(L\log L)^{1,\kappa}(\mathbb R^n)$ and $\|\cdot\|_{(L\log L)^{1,\kappa}}$. This new space was defined and investigated in \cite{sa} and \cite{lida}. By definition, it is obvious that for $(p,\kappa)\in[1,\infty)\times(0,1)$, $(L\log L)^{p,\kappa}(\mathbb R^n)\subset L^{p,\kappa}(\mathbb R^n)$, and
\begin{equation}\label{lplog2}
\|f\|_{L^{p,\kappa}}\leq \|f\|_{(L\log L)^{p,\kappa}}.
\end{equation}
In addition, it is easy to verify that for any $1\leq p<\infty$ and $0<\kappa<1$,
\begin{equation}\label{pp2}
\big\||f|^p\big\|_{(L\log L)^{1,\kappa}}\leq \|f\|^p_{(L\log L)^{p,\kappa}}.
\end{equation}
Recall that the following statement is true:
\begin{equation}\label{Mpp2}
\|M(f)\|_{L^{1,\kappa}}\leq C\|f\|_{(L\log L)^{1,\kappa}},
\end{equation}
for all $f\in (L\log L)^{1,\kappa}(\mathbb R^n)$ with $0<\kappa<1$ (see \cite[Proposition 6.1]{lida} and \cite[Corollary 2.21]{sa}). Based on the results mentioned above, we can prove the following result for the endpoint case $p=s'$.
\begin{thm}\label{thm8}
Let $\Omega\in L^s(\mathbf{S}^{n-1})$ with $1<s\leq\infty$. Let $0<\alpha<n$, $1\leq p<q<\infty$, $0<\kappa<1-{(\alpha p)}/n$ and $1/q=1/p-\alpha/{n(1-\kappa)}$. If $s'=p<n/{\alpha}$, then the inequality
\begin{equation*}
\|T_{\Omega,\alpha}(f)\|_{L^{q,\kappa}}\lesssim\|f\|_{(L\log L)^{p,\kappa}}
\end{equation*}
holds.
\end{thm}
\begin{proof}[Proof of Theorem \ref{thm8}]
Taking into account \eqref{wanghua2}, we obtain
\begin{equation*}
\begin{split}
\|T_{\Omega,\alpha}(f)\|_{L^{q,\kappa}}
&\lesssim\sup_{B\subseteq\mathbb R^n}\bigg(\frac{1}{m(B)^{\kappa}}\int_{B}\big|M_{s'}f(x)\big|^{p}dx\bigg)^{1/q} \cdot\big[\|f\|_{L^{p,\kappa}}\big]^{1-p/q}\\
&=\sup_{B\subseteq\mathbb R^n}\bigg(\frac{1}{m(B)^{\kappa}}\int_{B}\big|M(|f|^{p})(x)\big|\,dx\bigg)^{1/{q}}
\cdot\big[\|f\|_{L^{p,\kappa}}\big]^{1-p/q}.
\end{split}
\end{equation*}
It is concluded from \eqref{lplog2}, \eqref{pp2} and \eqref{Mpp2} that
\begin{equation*}
\begin{split}
\|T_{\Omega,\alpha}(f)\|_{L^{q,\kappa}}&\leq C\big\||f|^p\big\|_{(L\log L)^{1,\kappa}}^{1/{q}}\cdot\big[\|f\|_{L^{p,\kappa}}\big]^{1-p/q}\\
&\leq C\big[\|f\|_{(L\log L)^{p,\kappa}}\big]^{p/q}\cdot\big[\|f\|_{(L\log L)^{p,\kappa}}\big]^{1-p/q}=C\|f\|_{(L\log L)^{p,\kappa}}.
\end{split}
\end{equation*}
We are done.
\end{proof}
By using the same procedure as in the proof of Theorems \ref{thm9} and \ref{thm10}, the following results can also be obtained for the critical case $\kappa=1-{(\alpha p)}/n$(see unweighted version of Theorems 2.2 and 3.3 in \cite{wang2}).
\begin{thm}\label{thm5}
Suppose that $\Omega$ satisfies the $L^s$-Dini smoothness condition $\mathfrak{D}_{s}$ with $1<s\le\infty$. Let $0<\alpha<n$, $s'\leq p<n/{\alpha}$, $1/q=1/p-{\alpha}/n$ and $\kappa=p/q=1-{(\alpha p)}/n$. Then the inequality
\begin{equation*}
\|T_{\Omega,\alpha}(f)\|_{\mathrm{BMO}}\lesssim\|f\|_{L^{p,\kappa}}
\end{equation*}
holds.
\end{thm}

\begin{thm}\label{thm6}
Suppose that $\Omega\in L^s(\mathbf{S}^{n-1})$ with $1<s\le\infty$. Let $0<\alpha<n$, $s'\leq p<n/{\alpha}$, $1/q=1/p-{\alpha}/n$ and $\kappa=p/q=1-{(\alpha p)}/n$. Then the inequality
\begin{equation*}
\|M_{\Omega,\alpha}(f)\|_{L^{\infty}}\lesssim\|f\|_{L^{p,\kappa}}
\end{equation*}
holds.
\end{thm}
We now claim that the inclusion relation
\begin{equation*}
L^{p,\infty}(\mathbb R^n)\hookrightarrow L^{q,\kappa}(\mathbb R^n)
\end{equation*}
holds with $1\leq q<p$ and $\kappa=1-q/p$. Moreover, for any $f\in L^{p,\infty}(\mathbb R^n)$ with $1<p<\infty$,  there exists a constant $C=C(p,q)>0$ such that
\begin{equation}\label{cpq}
\|f\|_{L^{q,\kappa}}\leq C(p,q)\|f\|_{L^{p,\infty}}
\end{equation}
with $1\leq q<p$ and $\kappa=1-q/p$. In fact, for any ball $B$ in $\mathbb R^n$, we have
\begin{equation*}
\begin{split}
\|f\cdot\chi_{B}\|_{L^q}&=\bigg(\int_{B}|f(x)|^q\,dx\bigg)^{1/q}\\
&=\bigg(\int_0^{\infty}q\lambda^{q-1}m\big(\big\{x\in B:|f(x)|>\lambda\big\}\big)\,d\lambda\bigg)^{1/q}\\
&=\bigg(\int_{0}^{\sigma}\cdots+\int_{\sigma}^{\infty}\cdots\bigg)^{1/q},
\end{split}
\end{equation*}
where $\sigma>0$ is a constant to be chosen later. Clearly, the first integral in the bracket is bounded by
\begin{equation*}
\begin{split}
\int_0^{\sigma}q\lambda^{q-1}m(B)\,d\lambda=\sigma^q\cdot m(B).
\end{split}
\end{equation*}
Using the fact that $q-p<0$, the second integral in the bracket is then controlled by
\begin{equation*}
\begin{split}
&\int_{\sigma}^{\infty}q\lambda^{q-1}m\big(\big\{x\in\mathbb R^n:|f(x)|>\lambda\big\}\big)\,d\lambda\\
&\leq\int_{\sigma}^{\infty}q\lambda^{q-1}\Big(\frac{\|f\|_{L^{p,\infty}}}{\lambda}\Big)^pd\lambda\\
&=\frac{q}{p-q}\Big(\sigma^{q-p}\cdot\|f\|^p_{L^{p,\infty}}\Big).
\end{split}
\end{equation*}
We now choose $\sigma$ such that $\sigma^q\cdot m(B)=\sigma^{q-p}\cdot\|f\|^p_{L^{p,\infty}}$. That is,
\begin{equation*}
\sigma:=\frac{\|f\|_{L^{p,\infty}}}{m(B)^{1/p}}.
\end{equation*}
Hence
\begin{equation*}
\begin{split}
\|f\cdot\chi_{B}\|_{L^q}&\leq\Big(1+\frac{q}{p-q}\Big)^{1/q}\Big(\sigma^q\cdot m(B)\Big)^{1/q}\\
&=\Big(\frac{p}{p-q}\Big)^{1/q}m(B)^{1/q-1/p}\|f\|_{L^{p,\infty}},
\end{split}
\end{equation*}
which is equivalent to that
\begin{equation*}
\frac{1}{m(B)^{1/q-1/p}}\|f\cdot\chi_{B}\|_{L^q}\leq \Big(\frac{p}{p-q}\Big)^{1/q}\|f\|_{L^{p,\infty}}.
\end{equation*}
This gives the desired estimate \eqref{cpq} and completes the proof of the claim, by taking the supremum over all balls $B$ in $\mathbb R^n$.
As a special case of \eqref{cpq}, when $0<\alpha<n$ and $p^{*}=n/{\alpha}$, one has
\begin{equation}\label{cpq2}
\|f\|_{L^{p,\kappa}}\leq C\|f\|_{L^{p^{*},\infty}}
\end{equation}
with $1\leq p<n/{\alpha}$ and $\kappa=1-{(\alpha p)}/n$.

Finally, in view of \eqref{cpq2}, as an immediate consequence of Theorems \ref{thm5} and \ref{thm6}, we have the following results.
\begin{cor}\label{cor5}
Suppose that $\Omega$ satisfies the $L^s$-Dini smoothness condition $\mathfrak{D}_{s}$ with $s>n/{(n-\alpha)}$. Let $0<\alpha<n$ and $p=n/{\alpha}$. Then the inequality
\begin{equation*}
\|T_{\Omega,\alpha}(f)\|_{\mathrm{BMO}}\lesssim\|f\|_{L^{p,\infty}}
\end{equation*}
holds.
\end{cor}

\begin{cor}\label{cor6}
Suppose that $\Omega\in L^s(\mathbf{S}^{n-1})$ with $s>n/{(n-\alpha)}$. Let $0<\alpha<n$ and $p=n/{\alpha}$. Then the inequality
\begin{equation*}
\|M_{\Omega,\alpha}(f)\|_{L^{\infty}}\lesssim\|f\|_{L^{p,\infty}}
\end{equation*}
holds.
\end{cor}
\section{Applications to several inequalities on $\mathbb R^n$}
\subsection{Hardy--Littlewood--Sobolev inequalities}
The classical Hardy--Littlewood--Sobolev inequality on $\mathbb R^n$ says that for any $0<\lambda<n$ and all measurable functions $(f,g)\in L^p\times L^q$, we have
\begin{equation}\label{HLS}
\bigg|\int_{\mathbb R^n}\int_{\mathbb R^n}\frac{f(x)\cdot g(y)}{|x-y|^{\lambda}}\,dxdy\bigg|\leq C({p,q,\lambda,n})\|f\|_{L^p}\cdot\|g\|_{L^q},
\end{equation}
whenever $1<p,q<\infty$ and $1/p+1/q+\lambda/n=2$. This inequality appears in many areas of analysis, often in their dual forms as Sobolev inequalities and Onofri inequalities. For more information about the sharp constants and the optimizers in \eqref{HLS}, the reader is referred to \cite{li}.

By the results shown in Section \ref{sec21}, we can extend \eqref{HLS} to the following:
\begin{thm}\label{31}
Suppose that $\Omega\in L^s(\mathbf{S}^{n-1})$ with $1<s\leq\infty$. Let $0<\lambda<n$, $s'<p,q<\infty$ and $1/p+1/q+\lambda/n=2$. Then for any $f\in L^p(\mathbb R^n)$ and $g\in L^q(\mathbb R^n)$, there exists a constant $C>0$ independent of $f$ and $g$ such that
\begin{equation}\label{HLSWang}
\bigg|\int_{\mathbb R^n}\int_{\mathbb R^n}\frac{\Omega(x-y)f(x)\cdot g(y)}{|x-y|^{\lambda}}\,dxdy\bigg|
\leq C\|f\|_{L^p}\cdot\|g\|_{L^q}.
\end{equation}
\end{thm}

\begin{proof}
Using H\"{o}lder's inequality, we have
\begin{equation*}
\begin{split}
&\bigg|\int_{\mathbb R^n}\int_{\mathbb R^n}\frac{\Omega(x-y)f(x)\cdot g(y)}{|x-y|^{\lambda}}\,dxdy\bigg|\\
&=\bigg|\int_{\mathbb R^n}f(x)\bigg(\int_{\mathbb R^n}\frac{\Omega(x-y)g(y)}{|x-y|^{\lambda}}\,dy\bigg)dx\bigg|\\
&=\bigg|\int_{\mathbb R^n}f(x)\cdot T_{\Omega,n-\lambda}(g)(x)\,dx\bigg|\\
&\leq\|f\|_{L^p}\cdot\|T_{\Omega,n-\lambda}(g)\|_{L^{p'}}.
\end{split}
\end{equation*}
Observe that $1/{p'}=1/q-{(n-\lambda)}/n$ and $s'<q<n/{(n-\lambda)}$. By the second part of Theorem \ref{thm1}, we get
\begin{equation*}
\|T_{\Omega,n-\lambda}(g)\|_{L^{p'}}\leq C\|g\|_{L^q}.
\end{equation*}
It is easy to see that (by changing the order of integration)
\begin{equation*}
\int_{\mathbb R^n}f(x)\cdot T_{\Omega,n-\lambda}(g)(x)\,dx=\int_{\mathbb R^n}T_{\widetilde{\Omega},n-\lambda}(f)(y)\cdot g(y)\,dy,
\end{equation*}
where $\widetilde{\Omega}(x):=\Omega(-x)$. Obviously, under the same assumptions as in Theorem \ref{thm1}, the conclusion of Theorem \ref{thm1} also holds for $\widetilde{\Omega}(x)$. This fact along with H\"{o}lder's inequality yields
\begin{equation*}
\begin{split}
\bigg|\int_{\mathbb R^n}f(x)\cdot T_{\Omega,n-\lambda}(g)(x)\,dx\bigg|
&=\bigg|\int_{\mathbb R^n}T_{\widetilde{\Omega},n-\lambda}(f)(y)\cdot g(y)\,dy\bigg|\\
&\leq \|T_{\widetilde{\Omega},n-\lambda}(f)\|_{L^{q'}}\|g\|_{L^q}\\
&\leq C\|f\|_{L^p}\cdot\|g\|_{L^q},
\end{split}
\end{equation*}
where in the last inequality we have used the fact that $1/{q'}=1/p-{(n-\lambda)}/n$ and $s'<p<n/{(n-\lambda)}$. From this, the desired inequality follows.
\end{proof}
Observe that \eqref{HLSWang} may be rewritten as
\begin{equation*}
\|f\cdot T_{\Omega,n-\lambda}(g)\|_{L^1}\leq C\|f\|_{L^p}\cdot\|g\|_{L^q},
\end{equation*}
for any $f\in L^p(\mathbb R^n)$ and $g\in L^q(\mathbb R^n)$ with $p,q>s'$. Equivalently, when $0<\alpha<n$ and $1/p+1/q=1+\alpha/n$, we know that
\begin{equation}\label{L1norm}
\|f\cdot T_{\Omega,\alpha}(g)\|_{L^1}\leq C\|f\|_{L^p}\cdot\|g\|_{L^q},
\end{equation}
for any $f\in L^p(\mathbb R^n)$ and $g\in L^q(\mathbb R^n)$ with $p,q>s'$. Let us now consider more general situation. Recall that the following two estimates hold (see \cite[p.11 and p.16]{grafakos}).
\begin{lem}\label{weakHold}
Let $1\leq p,q<\infty$ and $0<r<\infty$ such that $1/r=1/p+1/q$. Then we have
\begin{enumerate}
  \item (H\"older's inequality for $L^p$ spaces) $\|F\cdot G\|_{L^r}\leq \|F\|_{L^p}\cdot\|G\|_{L^q};$
  \item (H\"older's inequality for weak $L^p$ spaces) $\|F\cdot G\|_{L^{r,\infty}}\leq C\|F\|_{L^p}\cdot\|G\|_{L^{q,\infty}}$.
\end{enumerate}
\end{lem}
Arguing as in the proof of Theorem \ref{31}, we can show that the $L^1$ norm in \eqref{L1norm} can be replaced by the general $L^r$ norm (or $L^{r,\infty}$ norm).
\begin{thm}\label{thm15}
Suppose that $\Omega\in L^s(\mathbf{S}^{n-1})$ with $1<s\leq\infty$. Let $0<\alpha<n$, $s'\leq p,q<\infty$ and $1/p+1/q=1/r+\alpha/n$ with $0<r<\min\{p,q\}$.

$(i)$ If $s'<p,q<\infty$, then for any $f\in L^p(\mathbb R^n)$ and $g\in L^q(\mathbb R^n)$, there exists a constant $C>0$ independent of $f$ and $g$ such that
\begin{equation*}
\|f\cdot T_{\Omega,\alpha}(g)\|_{L^{r}}\leq C\|f\|_{L^p}\cdot\|g\|_{L^q}.
\end{equation*}

$(ii)$ If $s'=p$ and $s'<q<\infty$(or $s'=q$ and $s'<p<\infty$), then for any $f\in L^p(\mathbb R^n)$ and $g\in L^q(\mathbb R^n)$, there exists a constant $C>0$ independent of $f$ and $g$ such that
\begin{equation*}
\|f\cdot T_{\Omega,\alpha}(g)\|_{L^{r,\infty}}\leq C\|f\|_{L^p}\cdot\|g\|_{L^q}.
\end{equation*}
\end{thm}

\begin{thm}\label{thm16}
Suppose that $\Omega\in L^s(\mathbf{S}^{n-1})$ with $1<s\leq\infty$. Let $0<\alpha<n$, $s'\leq p,q<\infty$ and $1/p+1/q=1/r+\alpha/n$ with $0<r<\min\{p,q\}$.

$(i)$ If $s'=p$ and $s'<q<\infty$, then for any given ball $B$ in $\mathbb R^n$, we have
\begin{equation*}
\|f\cdot T_{\Omega,\alpha}(g)\|_{L^{r}(B)}\leq C\|f\|_{L^p\log L(B)}\cdot\|g\|_{L^q(B)}.
\end{equation*}

$(ii)$ If $s'=q$ and $s'<p<\infty$, then for any given ball $B$ in $\mathbb R^n$, we have
\begin{equation*}
\|f\cdot T_{\Omega,\alpha}(g)\|_{L^{r}(B)}\leq C\|f\|_{L^p(B)}\cdot\|g\|_{L^q\log L(B)}.
\end{equation*}
\end{thm}
\begin{proof}
Theorem \ref{thm15} is a straightforward consequence of Theorem \ref{thm1} and Lemma \ref{weakHold}, and Theorem \ref{thm16} is a straightforward consequence of Theorem \ref{thm7} and Lemma \ref{weakHold}.
\end{proof}

\subsection{Olsen-type inequalities}
A classical result of Olsen \cite{ol} states that
\begin{equation*}
\|f\cdot I_{\alpha}(g)\|_{L^{r,\kappa}}\leq C({p,q,\kappa,n})\|f\|_{L^{p,\kappa}}\cdot\|g\|_{L^{q,\kappa}},
\end{equation*}
under certain conditions on the parameters $p,q,r,\kappa$. Olsen considered this type of inequality to investigate the Schr\"{o}dinger equation. Later this inequality(also known as the trace inequality) was studied and sharpened by many authors. For more related results, the reader is referred to \cite{ol}, \cite{lida} and \cite{sa}. There is an analogue of H\"older's inequality for Morrey spaces and weak Morrey spaces. Let $0<\kappa<1$, $1\leq p,q<\infty$ and $0<r<\infty$ with $1/p+1/q=1/r$. By using H\"{o}lder's inequality for $L^p$ spaces, we get
\begin{equation}\label{dot}
\begin{split}
\bigg(\int_{B}|F(y)\cdot G(y)|^r\,dy\bigg)^{1/r}
&\leq\bigg(\int_{B}|F(y)|^p\,dy\bigg)^{1/p}\bigg(\int_{B}|G(y)|^q\,dy\bigg)^{1/q}\\
&\leq\big(\|F\|_{L^{p,\kappa}}m(B)^{\kappa/p}\big)\cdot\big(\|G\|_{L^{q,\kappa}}m(B)^{\kappa/q}\big)\\
&=\big(\|F\|_{L^{p,\kappa}}\cdot\|G\|_{L^{q,\kappa}}\big)m(B)^{\kappa/r}.
\end{split}
\end{equation}
Since \eqref{dot} is independent of the choice of the ball $B$, we conclude that $F\cdot G$ is in $L^{r,\kappa}(\mathbb R^n)$ and
\begin{equation*}
\|F\cdot G\|_{L^{r,\kappa}}\leq \|F\|_{L^{p,\kappa}}\cdot\|G\|_{L^{q,\kappa}}.
\end{equation*}
A similar argument (by H\"{o}lder's inequality for weak $L^p$ spaces) gives us the following:
\begin{equation*}
\|F\cdot G\|_{WL^{r,\kappa}}\leq C\|F\|_{L^{p,\kappa}}\cdot\|G\|_{WL^{q,\kappa}}.
\end{equation*}
\begin{lem}\label{weakMHold}
Let $0<\kappa<1$, $1\leq p,q<\infty$ and $0<r<\infty$ such that $1/r=1/p+1/q$. Then we have
\begin{enumerate}
  \item (H\"older's inequality for Morrey spaces) $\|F\cdot G\|_{L^{r,\kappa}}\leq \|F\|_{L^{p,\kappa}}\cdot\|G\|_{L^{q,\kappa}};$
  \item (H\"older's inequality for weak Morrey spaces) $\|F\cdot G\|_{WL^{r,\kappa}}\leq C\|F\|_{L^{p,\kappa}}\cdot\|G\|_{WL^{q,\kappa}}$.
\end{enumerate}
\end{lem}
It should be pointed out that the first part of Lemma \ref{weakMHold} has been proved by Olsen \cite{ol} if $1\leq r<\infty$. In view of Lemma \ref{weakMHold} and Theorems \ref{thm2} and \ref{thm8}, we can also prove the following Olsen-type inequalities.
\begin{thm}\label{thm17}
Suppose that $\Omega\in L^s(\mathbf{S}^{n-1})$ with $1<s\leq\infty$. Let $0<\alpha<n$, $s'\leq p,q<\infty$, and $1/p+1/q=1/r+\alpha/{n(1-\kappa)}$ with $0<r<\min\{p,q\}$ and $0<\kappa<1$.

$(i)$ If $s'<p,q<\infty$, then for any $f\in L^{p,\kappa}(\mathbb R^n)$ and $g\in L^{q,\kappa}(\mathbb R^n)$, there exists a constant $C>0$ independent of $f$ and $g$ such that
\begin{equation*}
\|f\cdot T_{\Omega,\alpha}(g)\|_{L^{r,\kappa}}\leq C\|f\|_{L^{p,\kappa}}\cdot\|g\|_{L^{q,\kappa}}.
\end{equation*}

$(ii)$ If $s'=p$ and $s'<q<\infty$(or $s'=q$ and $s'<p<\infty$), then for any $f\in L^{p,\kappa}(\mathbb R^n)$ and $g\in L^{q,\kappa}(\mathbb R^n)$, there exists a constant $C>0$ independent of $f$ and $g$ such that
\begin{equation*}
\|f\cdot T_{\Omega,\alpha}(g)\|_{WL^{r,\kappa}}\leq C\|f\|_{L^{p,\kappa}}\cdot\|g\|_{L^{q,\kappa}}.
\end{equation*}
\end{thm}
\begin{proof}
The condition $1/p+1/q=1/r+\alpha/{n(1-\kappa)}$ with $0<r<\min\{p,q\}$ implies that $0<\kappa<1-{(\alpha p)}/n$ and $0<\kappa<1-{(\alpha q)}/n$.
\begin{itemize}
  \item If $s'<p,q<\infty$, by using Lemma \ref{weakMHold}(1), then we have
\begin{equation*}
\|f\cdot T_{\Omega,\alpha}(g)\|_{L^{r,\kappa}}\leq \|f\|_{L^{p,\kappa}}\cdot\|T_{\Omega,\alpha}(g)\|_{L^{q^{*},\kappa}},
\end{equation*}
where the positive number $q^{*}$ satisfies $1/r=1/p+1/{q^{*}}$. Notice that $1/{q^{*}}=1/q-\alpha/{n(1-\kappa)}$. Thus, by Theorem \ref{thm2}, we further have
\begin{equation*}
\|f\cdot T_{\Omega,\alpha}(g)\|_{L^{r,\kappa}}\leq C\|f\|_{L^{p,\kappa}}\cdot\|g\|_{L^{q,\kappa}},
\end{equation*}
whenever $s'<q<n/{\alpha}$ and $0<\kappa<1-{(\alpha q)}/n$. In addition, it is easy to see that
\begin{equation}\label{dualform}
\|f\cdot T_{\Omega,\alpha}(g)\|_{L^{r,\kappa}}=\|T_{\widetilde{\Omega},\alpha}(f)\cdot g\|_{L^{r,\kappa}},
\end{equation}
where $\widetilde{\Omega}(x)$ is the same as in the previous proof. Obviously, under the same conditions of Theorem \ref{thm2}, the conclusion also holds for the fractional integral $T_{\widetilde{\Omega},\alpha}$ related to $\widetilde{\Omega}(x)$. It follows from \eqref{dualform} and Lemma \ref{weakMHold}(1) that
\begin{equation*}
\|f\cdot T_{\Omega,\alpha}(g)\|_{L^{r,\kappa}}\leq\|T_{\widetilde{\Omega},\alpha}(f)\|_{L^{p^{*},\kappa}}\cdot\|g\|_{L^{q,\kappa}},
\end{equation*}
where the positive number $p^{*}$ satisfies $1/r=1/{p^{*}}+1/q$. Also notice that $1/{p^{*}}=1/p-\alpha/{n(1-\kappa)}$. Therefore, by using Theorem \ref{thm2} again,
\begin{equation*}
\|f\cdot T_{\Omega,\alpha}(g)\|_{L^{r,\kappa}}\leq C\|f\|_{L^{p,\kappa}}\cdot\|g\|_{L^{q,\kappa}},
\end{equation*}
whenever $s'<p<n/{\alpha}$ and $0<\kappa<1-{(\alpha p)}/n$. This shows $(i)$.
  \item If $s'=p$ and $s'<q<\infty$, from \eqref{dualform}, Lemma \ref{weakMHold}(2) and Theorem \ref{thm2}, it then follows that
\begin{equation*}
\begin{split}
\|f\cdot T_{\Omega,\alpha}(g)\|_{WL^{r,\kappa}}&=\|T_{\widetilde{\Omega},\alpha}(f)\cdot g\|_{WL^{r,\kappa}}\\
&\leq\|T_{\widetilde{\Omega},\alpha}(f)\|_{WL^{p^{*},\kappa}}\cdot\|g\|_{L^{q,\kappa}}\\
&\leq C\|f\|_{L^{p,\kappa}}\cdot\|g\|_{L^{q,\kappa}}.
\end{split}
\end{equation*}
As for the case $s'=q$ and $s'<p<\infty$, the proof is the same. This shows $(ii)$ and the proof is complete.
\end{itemize}
\end{proof}

\begin{thm}\label{thm18}
Suppose that $\Omega\in L^s(\mathbf{S}^{n-1})$ with $1<s\leq\infty$. Let $0<\alpha<n$, $s'\leq p,q<\infty$ and $1/p+1/q=1/r+\alpha/{n(1-\kappa)}$ with $0<r<\min\{p,q\}$ and $0<\kappa<1$.

$(i)$ If $s'=q$ and $s'<p<\infty$, then we have
\begin{equation*}
\|f\cdot T_{\Omega,\alpha}(g)\|_{L^{r,\kappa}}\leq C\|f\|_{L^{p,\kappa}}\cdot\|g\|_{(L\log L)^{q,\kappa}}.
\end{equation*}

$(ii)$ If $s'=p$ and $s'<q<\infty$, then we have
\begin{equation*}
\|f\cdot T_{\Omega,\alpha}(g)\|_{L^{r,\kappa}}\leq C\|f\|_{(L\log L)^{p,\kappa}}\cdot\|g\|_{L^{q,\kappa}}.
\end{equation*}
\end{thm}
\begin{proof}
$(i)$ It follows immediately from Lemma \ref{weakMHold}(1) that
\begin{equation*}
\|f\cdot T_{\Omega,\alpha}(g)\|_{L^{r,\kappa}}\leq \|f\|_{L^{p,\kappa}}\cdot\|T_{\Omega,\alpha}(g)\|_{L^{q^{*},\kappa}},
\end{equation*}
where $1/r=1/p+1/{q^*}$. Note that $1/{q^*}=1/q-\alpha/{n(1-\kappa)}$. Thus, by Theorem \ref{thm8}, we deduce that
\begin{equation*}
\|f\cdot T_{\Omega,\alpha}(g)\|_{L^{r,\kappa}}\leq C\|f\|_{L^{p,\kappa}}\cdot\|g\|_{(L\log L)^{q,\kappa}}.
\end{equation*}
This proves $(i)$. $(ii)$ can be proved in the same manner and the proof is complete.
\end{proof}

\bibliographystyle{elsarticle-harv}
\bibliography{<your bibdatabase>}

\end{document}